\documentclass[11pt,a4paper,reqno]{article}
\author{Mara Ungureanu}
\title{Geometry of intersections of some secant varieties to algebraic curves}
\date{}

\AtEndDocument{\medskip
\noindent\begin{footnotesize}\textsc{Albert-Ludwigs-Universit\"at Freiburg, Mathematisches Institut, Abteilung Reine Mathematik, Ernst-Zermelo-Str 1, 79104 Freiburg}
\smallskip\\
\textit{E-mail address:} \texttt{mara.ungureanu@math.uni-freiburg.de}
\end{footnotesize}}

\usepackage{amsmath, amssymb, amsthm,mathrsfs}
\usepackage{enumitem}
\usepackage{color}
\usepackage{mathtools}
\usepackage[a4paper,left=3cm,right=3cm,bottom=4cm]{geometry}
\usepackage{lmodern}
\usepackage{indentfirst}
\usepackage{tikz}
\usetikzlibrary{matrix,arrows}
\usepackage{epstopdf}
\usepackage{float}
\usepackage{microtype}

\usepackage[utf8]{inputenc}
\usepackage[T1]{fontenc}

\newcommand{\id}{\text{Id}}
\newcommand{\grd}{g^r_d}

\newcommand{\grdop}[2]{g^{#1}_{#2}}
\DeclareMathOperator{\im}{\text{im}}

\DeclareMathOperator{\rk}{rk}
\DeclareMathOperator{\oo}{\mathcal{O}}

\DeclareMathOperator{\p}{\mathbb{P}}
\DeclareMathOperator{\pic}{\text{Pic}}
\DeclareMathOperator{\ord}{\text{ord}}
\newcommand{\Grd}{G^r_d(C)}
\newcommand{\Wrd}{W^r_d(C)}
\newcommand{\Crd}{C^r_d}

\newtheorem{thm}{Theorem}[section] 

\newtheorem{prop}[thm]{Proposition}

\theoremstyle{definition}

\theoremstyle{definition}

\theoremstyle{remark}
\newtheorem{rem}{Remark}[section]

\begin{document}

\maketitle

\begin{abstract}
For a smooth projective curve, the cycles of subordinate or, more generally, secant divisors to a given linear series are among some of the most studied objects in classical enumerative geometry.  We consider the intersection of two such cycles corresponding to secant divisors of two different linear series on the same curve and investigate the validity of the enumerative formulas counting the number of divisors in the intersection.  We study some interesting cases, with unexpected transversality properties, and establish a general method to verify when this intersection is empty.
\end{abstract}

\section{Introduction}
The study of the ways in which an algebraic curve can be mapped into projective space has been a fruitful avenue of research for algebraic geometers since the 19th century.  Not only does it allow for a better understanding of the relationship between intrinsic and extrinsic properties of projective curves, but it has also led to important results pertaining to the birational geometry of the moduli space of abstract curves and it has found modern incarnations in the subject of stable maps and its applications.

Moreover, the subject provides a rich source of enumerative questions.  One of the most basic such problems is to determine the number of singularities that occur on the image curve $f(C)$ under a mapping $f:C\rightarrow\mathbb{P}^r$ where $C$ is a smooth algebraic curve.  
An elementary example thereof is the calculation of the number of double points on the image curve $f(C)$  under a birational map to the quadric surface $\p^1 \times \p^1$.  Assuming the curve $C$ has arithmetic genus $g$ and bidegree $(d_1,d_2)$, the adjunction formula tells us that there are exactly
\begin{equation}\label{eq:adjunction}
\nu=(d_1-1)(d_2-1)-g
\end{equation} 
ordinary double points.

This problem can be reformulated from the point of view of a class computation for \textit{intersections of incidence varieties} as follows: the map 
\[ C\rightarrow\p^1\times\p^1 \]
is given by a pair of pencils $l_1$ and $l_2$ of degree $d_1$ and $d_2$, respectively, on $C$ and the double points correspond to pairs of points $(p_1,p_2)$ common to both linear series, i.e.~an effective divisor $D=p_1+p_2$ on the curve $C$ such that
\[\dim (l_1 - D)\geq 0 \text{ and } \dim (l_2-D)\geq 0.\]
In other words, as we shall make precise below, the divisor $D$ belongs to the \textit{incidence varieties} of both pencils $l_1$ and $l_2$ and
the enumerative problem becomes thus the problem of counting the number of divisors in the intersection of these two incidence varieties.
One may therefore obtain the count by simply computing the fundamental class of the intersection.  This computational approach has the advantage of being immediately generalisable to maps given by linear series of any degree and dimension, but as long as the geometry of the intersection is not known (i.e. whether it is smooth, reduced, of expected dimension, etc.), the result of the enumerative calculation may not be considered meaningful. 

The purpose of this paper is to study the geometry of such intersections of incidence (or more generally secant) varieties to algebraic curves, with a focus on issues of transversality of intersection.

\bigskip

In order to state the precise results, we introduce some terminology.
Let $C$ be a general curve of genus $g$ equipped with a linear series $l$ of effective divisors of degree $d$ and dimension $r$.  Such a linear series is called a $\grd$.  Let $e\leq d$ be a positive integer and denote by $C_e$ the $e$-th symmetric product of the curve.  We set  
\[ \Gamma_e(l):=\{ D\in C_e \mid D'-D\geq 0\text{ for some }D'\in l \}\subset C_e\]
or equivalently
\[ \Gamma_e(l):=\{ D\in C_e \mid \dim(l-D)\geq 0\}\subset C_e\]
to be the \textit{incidence variety} of all effective divisors of degree $e$ that are subordinate to the linear series $l$.

As a subspace of $C_e$, the space $\Gamma_e(l)$ has the structure of a degeneracy locus so it is indeed a variety and it is easy to see that it has expected dimension $r$ for an arbitrary linear series $l$ of type $\grd$.  We explain this in more detail in the course of the paper (see Section \ref{sec:prelimsec}).

Suppose the divisor $D$ belongs to an intersection $\Gamma_e(l_1)\cap\Gamma_e(l_2)$ of incidence varieties corresponding to two different linear series $l_1$, $l_2$ on the same curve $C$.  This of course imposes a stronger condition on $D$ inside $\Gamma_e(l_1)$ (or $\Gamma_e(l_2)$) than the one from the definition of incidence varieties.  This stronger condition, depending on the geometric situation, may for example give a higher bound on the dimension of the linear series $l_1-D$ (or $l_2-D$), which in turn means that $D$ should belong to a certain subspace of $\Gamma_e(l_1)$ (or $\Gamma_e(l_2)$).  One way to keep track of this is by means of a generalisation of the notion of incidence varieties, namely that of \textit{secant varieties}: if $C$ is a general curve of genus $g$ endowed with a linear series $l$ of type $\grd$ and if $e$ and $f$ are positive integers such that $0\leq f <e\leq d$, then let 
\[V_e^{e-f}(l)=\{D\in C_e \mid \dim (l-D)\geq r-e+f \}\subset C_e\] 
be the \textit{secant variety} of effective divisors of degree $e$ which impose at most $e-f$ independent conditions on $l$.  
Equivalently, this space parametrises the $e$-secant $(e-f-1)$-planes to the curve $C$ embedded in $\p^r$ via $l$.  It is easy to see that incidence varieties are special cases of secant varieties, namely $\Gamma_e(l)=V_e^r(l)$ and $f=e-r$ and moreover that $V^{e-f}_e(l)\subset\Gamma_e(l)$.  

The cycle $V_e^{e-f}(l)$ of $C_e$ is also endowed with a degeneracy locus structure (so it is an actual variety) and it has expected dimension
\[ \exp\dim V_e^{e-f}(l)=e-f(r+1-e+f). \]
 It was proven by Farkas \cite{Fa2} that, if non-empty, its dimension is the expected one 
 for a general curve $C$ with a general series $l$ of type $\grd$.
 
 Unlike incidence varieties, the general secant varieties have a more complicated geometry.  This is illustrated by the fact that existence results for the variety $V_e^{e-f}(l)$ when 
 \[ e-f(r+1-e+f)\geq 0 \]
 are only known for some possible values of the parameters $g,r,d,e,f$: $V^e_{e-f}(l)\neq\emptyset$ for every curve $C$ of genus $g$ and $l=\grd$ such that $d\geq 2e-1$ and $e-f(r+1-e+f)\geq r-e+f$ (cf.~Theorem 1.2 of \cite{CM}); or for every $l=\grd$ with $g-d+r\leq 1$ if and only if $\rho(g,r-e+f,d-e)\geq 0$ (cf.~\cite{ACGH} pg. 356).

\bigskip

Furthermore, understanding of the geometry of secant varieties and their intersections is worthwile because they are interesting objects not only from the point of view of classical algebraic geometry, but also from a modern perspective.  For example, one may generalise the notion of secant varieties to nonsingular projective surfaces $S$ with a line bundle $L$.  If $|L|$ is a linear system of dimension $3m-2$ inducing a map $S\rightarrow\p^{3m-2}$, then the number of $m$-chords of dimension $m-2$ to the image of $S$ (i.e.~the cardinality of the secant variety $V^{m-1}_m(|L|)$) is given by the integral of the top Segre class
\[ \int_{S^{[m]}} s_{2m}(H^{[m]}), \]
where $S^{[m]}$ is the Hilbert scheme of points of $S$ carrying a tautological rank-$m$ bundle
$H^{[m]}$.  Such Segre classes play a basic role in the Donaldson-Thomas counting of sheaves and appeared first in the algebraic study of Donaldson invariants via the moduli space of rank-2 bundles on $S$ \cite{T}.  The exact result of the integral is the subject of Lehn's conjecture \cite{L} that states that it can be expressed as a polynomial of degree $m$ in the four variables
\[ H^2,\  H\cdot K_S,\  K_S^2,\  c_2(S). \]
For a proof of this conjecture, see \cite{Ti} and for a generalisation in the case of K3 surfaces see \cite{MOP}.

\bigskip

To come back to our motivating enumerative problem, consider the following setup: equip the general curve $C$ of genus $g$ with two complete linear series $l_1=\grdop{r_1}{d_1}$ and $l_2=\grdop{r_2}{d_2}$ and let 
\begin{align*}
&\Gamma_e(l_1)=\{ D\in C_e \mid \dim(l_1-D)\geq 0 \},\\
&\Gamma_e(l_2)=\{ D\in C_e \mid \dim(l_2-D)\geq 0 \},
\end{align*}
be the respective incidence varieties of dimensions $r_1$ and $r_2$, respectively.  As mentioned earlier, we are interested in counting the number of points in the intersection $\Gamma_e(l_1)\cap\Gamma_e(l_2)$, in the cases when we expect this space to consist of a finite number of points, i.e.~when
\[\dim\Gamma_e(l_1)+\dim\Gamma_e(l_2)=r_1+r_2=e.\]
In fact, in Chapter VIII, \S 3 of \cite{ACGH}, a class computation shows that in this case, the number is expected to be the coefficient of the monomial $t_1^{e-r_1} t_2^{e-r_2}$ in 
\begin{equation}\label{eq:countincidencegeneral}
(1+t_1)^{d_1-g-r_1} (1+t_2)^{d_2-g-r_2} (1+t_1+t_2)^g. 
\end{equation} 
Using this formula we immediately recover the number of double points of the image curve $f(C)$ of genus $g$ and bidegree $(d_1,d_2)$ under a birational map $f$ to the quadric surface $\p^1 \times\p^1$. Indeed, in this case  $r_1=r_2=1$ and $e=2$.  Thus, according to formula (\ref{eq:countincidencegeneral}), the number we are after is the coefficient of $t_1 t_2$ in
\[  (1+t_1)^{d_1-g-1}(1+t_2)^{d_2-g-1}(1+t_1+t_2)^g, \]
which is exactly $(d_1-1)(d_2-2)-g$, i.e.~the same count obtained by geometric methods in (\ref{eq:adjunction}).

While in this case the class computation gives the correct number of points, as verified by the geometric argument using the adjunction formula,  (\ref{eq:countincidencegeneral}) unfortunately also yields unexpected zero counts in some other cases.  
A priori it is not at all clear whether these zero counts correspond to an empty intersection of incidence varieties or to the degenerate situation of a positive-dimensional intersection.  Our main contribution in this context is the clarification of this issue in some interesting cases.
 More precisely, we prove the following
\begin{thm}\label{thm:incidence}
	Consider a general curve $C$ of genus $g$ equipped with arbitrary complete linear series $l_1=\grdop{r_1}{d_1}$, $l_2=\grdop{r_2}{d_2}=K_C-l_1$ with non-negative Brill-Noether numbers and so that $r_1,r_2\geq 0$, and let $e\leq \min{(d_1,d_2)}$ be a positive integer.  
		 If $e=r_1+r_2$, $d_2=r_1 + 2$ (or equivalently $r_2=1$), and the Brill-Noether number $\rho(g,r_1,d_1)$ vanishes, then the intersection $\Gamma_e(l_1)\cap\Gamma_e(l_2)$ is empty.	
\end{thm}

\medskip

The discussion above shows that another sensible direction is to consider directly the intersection of an incidence variety and a secant variety on a smooth general curve $C$, namely
\[ \Gamma_e(l_1) \cap V^{e-f}_e(l_2),\]
where $l_1=\grdop{r_1}{d_1}$ and $l_2=\grdop{r_2}{d_2}$ are linear series on $C$ and $e$ and $f$ are integers such that $0\leq f<e\leq \min{(d_1,d_2)}$.
We investigate the expected emptiness of the intersection when the sum of the dimensions of the two varieties $\Gamma_e(l_1)$ and $V^{e-f}_e(l_2)$ inside $C_e$ is less than $e$.  As in Theorem \ref{thm:incidence} we again focus on the case $l_2=K_C-l_1$.
Our main result in this context is:
\begin{thm}\label{thm:secant}
	Let $C$ be a general curve of genus $g$ equipped with an arbitrary complete linear series $l_1=\grdop{r_1}{d_1}$ such that $\rho(g,r_1,d_1)\geq 0$ and $\dim (K_C-l_1)\geq 0$.  If $f=1$ and
\begin{equation}\label{eq:a}
\dim \Gamma_e(l_1) +\exp \dim V_e^{e-f}(K_C-l_1)\leq e-\rho(g,r_1,d_1)-1,
\end{equation}  
then the intersection $\Gamma_e(l_1)\cap V_e^{e-f}(K_C-l_1)$ is empty.
\end{thm}
Note that if equality holds in (\ref{eq:a}) and $f=r_1+1+\rho(g,r_1,d_1)$, then
$\dim(K_C-l_1)-e+f=0$ and it follows that $V_e^{e-f}(K_C-l_1)=\Gamma_e(K_C-l_1)$. We discuss the geometric interpretation of this case  when $l_1$ is a pencil in Section \ref{sec:pencil}, Part I.

We prove Theorem \ref{thm:secant} by degeneration to a nodal curve using limit linear series and by exploiting an ingenious construction of Farkas (\cite{Fa2}). 
Furthermore, we provide in fact a method to check the emptiness of such intersections for any $f$, but the case $f=1$ seems to be the one with the most tractable computations.

In the course of the proof of Theorem \ref{thm:secant} we also find an interesting example that contradicts the expectation of non-emptiness of secant varieties $V_e^{e-f}(l)$, where $l=\grd$, when the expected dimension
\[\exp\dim V_e^{e-f}(l) = e-f(r+1-e+f) \geq 0. \]
 We explain this in Remark \ref{rem:interesting} of Section \ref{sec:proofincidsecinter}.

\bigskip

The paper is organised as follows: in Section \ref{sec:prelimsec} we establish some preliminary results on incidence and secant varieties and we describe their tangent space in Section \ref{sec:tangentspace}.  We prove Theorem \ref{thm:incidence} in Section \ref{sec:interincid}.  We then construct degenerations of secant varieties for families of curves with nodal fibres of compact type using limit linear series in Section \ref{sec:degenerationssecant} and we use them to prove Theorem \ref{thm:secant} in Section \ref{sec:proofincidsecinter}.

\section*{Acknowledgements}
This paper is part of my PhD thesis.  I would like to thank my advisor Gavril Farkas for his helpful suggestions.  I am also grateful to the referee for pointing out several inaccuracies in a previous version of this article and for their comments that lead to a significant improvement of the manuscript.

\section{Preliminaries on incidence and secant varieties}\label{sec:prelimsec}
We begin by fixing the notation.
Let $C$ be a smooth curve of genus $g$ and denote by $C_d$ its $d$-th symmetric product.  Furthermore, let $\Grd$ parametrise linear series of type $\grd$, i.e.
\[\Grd:=\{ l=(L,V) \mid L\in\pic^d(C), V\in G(r+1,H^0(C,L)) \}.\]  
Furthermore, given a divisor $D\in C_d$, denote by $|D|$ the complete linear system of effective divisors linearly equivalent to $D$.   Let $\Crd$ be the subvariety of $C_d$ parametrising effective divisors of degree $d$ on $C$ moving in a linear series of dimension at least $r$:  
\[ \Crd := \{ D\in C_d \mid \dim | D | \geq r \}, \]
and $\Wrd$ be the associated variety of complete linear series of degree $d$ and dimension at least $r$, i.e.
\[ \Wrd := \{ L\in\pic^d(C) \mid h^0(C,L)\geq r+1 \}\subseteq\pic^d(C).\]
We focus on Brill-Noether general curves, meaning that $\Grd$ is a smooth variety and its dimension is given by the Brill-Noether number
\begin{equation}\label{eq:brillnoether} 
\rho(g,r,d)=g-(r+1)(g-d+r)\geq 0. 
\end{equation}

\bigskip

Let $e,f$ be integers such that $0\leq f<e\leq d$.
As mentioned in the Introduction, incidence varieties are special cases of secant varieties $V^{e-f}_{e}(l)$ with $r-e+f=0$, namely $\Gamma_e(l)=V_e^r(l)$.

 Secant (and therefore incidence) varieties $V^{e-f}_e(l)$ of effective divisors of degree $e$ imposing at most $e-f$ conditions on $l$ have a degeneracy locus structure inside the symmetric product $C_e$, obtained as follows: let $\mathcal{E}=\oo_{C_e}\otimes V$ be the trivial vector bundle of rank $r+1$ on $C_e$ and
$\mathcal{F}_e(L):=\tau_*(\sigma^*L\otimes\oo_{\mathcal{U}})$ be the $e$-th secant bundle, where $\mathcal{U}$ is the universal divisor 
\[ \mathcal{U}=\{ (p,D)\mid D\in C_e \text{ and }p\in D\}\subset C\times C_e, \]
and $\sigma$, $\tau$ are the usual projections:
 \begin{figure}[H]\centering
  \begin{tikzpicture}
    \matrix (m) [matrix of math nodes,row sep=2em,column sep=1em,minimum width=1em]
  {
     & C \times C_e & \supset \mathcal{U} \\
     C & & C_e\\};
  \path[-stealth]
    (m-1-2) edge node [auto,swap] {$\sigma$} (m-2-1)
            edge node [auto]{$\tau$}  (m-2-3);
 \end{tikzpicture}
 \end{figure}
Let $\Phi:\mathcal{E}\rightarrow\mathcal{F}$ be the bundle morphism obtained by pushing down to $C_e$ the restriction $\sigma^*L	\rightarrow\sigma^*L\otimes\oo_{\mathcal{U}}$.  The space $V^{e-f}_e(l)$ is then the $(e-f)$-th degeneracy locus of $\Phi$, i.e.~where $\rk\Phi\leq e-f$. To see that this is indeed the case, note that fibrewise, the morphism $\Phi$ is given by the restriction:
\[ \Phi_D:H^0(C,L)\rightarrow H^0(C,L/L(-D)). \] 
 Now by definition, $D\in V^{e-f}_e(l)$ if and only if $\dim\ker\Phi_D=h^0(L-D)\geq r+1-e+f$, which is equivalent to the aforementioned condition $\rk\Phi\leq e-f$. 
The dimension estimate for $V^{e-f}_e(l)$ follows immediately from its degeneracy locus structure:
\[ \dim V^{e-f}_e(l)\geq e-(r+1-e+f)(e-e+f)=e-f(r+1-e+f). \]
In particular, 
\[ \dim\Gamma_e(l) \geq r. \]
On the other hand, since $D\in\Gamma_e(l)$ is equivalent to there existing a divisor $E\in l$ such that $E-D\geq 0$, and since the dimension of the locus of such divisors $E$ inside $l$ is at most $r$, we immediately have that
\[ \dim \Gamma_e(l)=r\]
for any linear series $l$ of type $\grd$ on $C$.
Using the Porteous formula, one obtains (see \cite{ACGH} Chapter VIII, Lemma 3.2) that the fundamental class of $\Gamma_e(l)$ is given by
\[ \gamma_e(l)=\sum_{j=0}^{e-r}\binom{d-g-r}{j}\frac{x^k \theta^{e-r-j}}{(e-r-j)!}, \]
where $\theta$ is the pullback of the fundamental class of the theta divisor to $C_e$ and $x$ is the class of the divisor $q+C_{e-1}\subset C_e$.  

To obtain formula (\ref{eq:countincidencegeneral}) giving the number (when expected to be finite) of divisors in the intersection
\[ \Gamma_e(l_1)\cap\Gamma_e(l_2), \]
where $l_1=\grdop{r_1}{d_1}$ and $l_2=\grdop{r_2}{d_2}$, one may compute the product
\[ \gamma_e(l_1)\gamma_e(l_2)\in H^{2e}(C_e,\mathbb{Z})\simeq\mathbb{Z}, \]
which, as shown in \cite{ACGH} Chapter VIII, pg. 343, yields the desired count.

Unfortunately, the situation is not so simple in the general case of secant varieties with $r-e+f>0$.  Indeed, the fundamental class of $V^{e-f}_e(l)$ has been computed by MacDonald and its expression is very complicated and thus of limited practical use, as can be seen in \cite{ACGH}, Chapter VIII, \S 4.
For a study of the dimension theory of secant varieties we refer the reader to \cite{Fa2}.
 
 In this paper we are concerned instead with the study of intersections of incidence and secant varieties on a given general smooth curve and with the geometric interpretation of some unexpected enumerative results that arise in this context.

\section{Infinitesimal study of secant varieties}\label{sec:tangentspace}
This section is dedicated to the infinitesimal study of secant varieties. More precisely, given a complete linear series $l$ we compute the tangent space of $V^{e-f}_e(l)$ at a point $D\in V^{e-f}_e(l)\setminus V^{e-f-1}_e(l)$ whose support consists of distinct points.  Using this we then write down a transversality condition for the intersection $\Gamma_e(l_1)\cap\Gamma_e(l_2)$ for two complete linear series $l_1$ and $l_2$.

In Section \ref{sec:prelimsec} we expressed $V^{e-f}_e(l)$ as a degeneracy locus of a map of vector bundles over $C_e$.  We find the tangent space by using a local description of such loci, arguing in the same spirit as in the tangent space computation for the variety $\Crd$ (see Lemma 1.5 in Chapter IV of \cite{ACGH}).

Let $l=(L,V)\in\Grd$ be a complete linear series and $D=\sum_{i=1}^n a_i p_i \in C_e$ where the points $p_i$ are distinct and the $a_i$ are positive integers such that $\sum_{i=1}^n a_i=e$. Choose a basis $\{s_0,\ldots,s_r\}$ of the vector space $H^0(C,L)$.    
Then $D\in V_e^{e-f}(l)$ if and only if the matrix below has rank at most $e-f$:
\[\phi(D):= \left(\begin{array}{ccc}
s_{01}(0) & \cdots & s_{r1}(0)\\
s'_{01}(0)&\cdots & s'_{r1}(0)\\
\vdots & & \vdots\\
\frac{1}{(a_1-1)!}s_{01}^{(a_1-1)}(0) & \cdots & \frac{1}{(a_1-1)!}s_{r1}^{(a_1-1)}(0)\\
s_{02}(0) & \cdots & s_{r2}(0)\\
s'_{02}(0)&\cdots & s'_{r2}(0)\\
\vdots & & \vdots\\
\frac{1}{(a_2-1)!}s_{02}^{(a_2-1)}(0) & \cdots & \frac{1}{(a_2-1)!}s_{r2}^{(a_2-1)}(0)\\
\vdots & & \vdots
\end{array}\right), \]
where $s_{ji}$ denotes the section $s_j$ in a local coordinate system centred around the point $p_i$, with $j\in\{0,\ldots,r\}$ and $i\in\{1,\ldots,n\}$.

Let $M$ denote the variety of $e\times(r+1)$ matrices.  We then interpret $\phi$ as an $M$-valued mapping defined in a neighbourhood of $D$.  Hence, by definition, $V_e^{e-f}(l)$ is in this neighbourhood the pullback via $\phi$ of the determinantal subvariety $M_{e-f}\subset M$ of all those matrices whose rank does not exceed $e-f$.  
Moreover, the Zariski tangent space to $V_e^{e-f}(l)$ at $D$ is the pullback of the tangent space to $M_{e-f}$ at $A=\phi(D)$ under the differential of $\phi$.  From Chapter 2, \S 2 of \cite{ACGH} we have that, if $A\notin M_{e-f-1}$ then:
\[ T_A (M_{e-f}) = \{ B\in M \mid B\cdot \ker A \subset A\cdot \mathbb{C}^{r+1} \}. \]
Since we are interested in the case $D\in V^{e-f}_e(l)\setminus V^{e-f-1}_e(l)$, the condition $A\notin M_{e-f-1}$ is satisfied.  Thus, for $A=\phi(D)$,
\begin{align*}
T_D(V_e^{e-f}(l)) &= \phi_*^{-1}(T_A(M_{e-f})) \\
&= \{ v\in T_D(C_e) \mid \phi_*(v)\cdot\ker A\subset \im A \}.
\end{align*}  
Recall that the tangent space $T_D C_e = H^0(C,\oo_D(D))$ and notice that by identifying the vector space $\mathbb{C}^{r+1}$ on which $A$ acts with $H^0(C,L)$ we see that $A$ is the matrix representing the restriction map
\[ \alpha_l : H^0(C,L)\rightarrow H^0(C,L\otimes \oo_D). \]
Hence we may write $\phi_*(v)\cdot \ker A$ as the image of $v\otimes H^0(C,L-D)$ under the cup-product homomorphism
\[ \beta_l: H^0(C,\oo_D(D))\otimes H^0(C,L-D) \rightarrow H^0(C,L\otimes \oo_D). \]
Thus the condition $\phi_*(v)\cdot\ker A\subset \im A$ is equivalent to
\[ \forall s\in H^0(C,L-D), \exists s'\in H^0(C,L)\colon \beta_l(v\otimes s)=\alpha_l(s'). \]
Denote by $\delta_l$ the coboundary mapping 
\[  H^0(C,L\otimes\oo_D)\rightarrow H^1(C,L-D) \]
and let 
\begin{align*}
\beta'_l:H^0(\oo_D(D))\otimes H^0(C,L-D)&\rightarrow H^1(C,L-D)\\
v\otimes s &\mapsto \delta_l (\beta_l(v\otimes s)).
\end{align*}
Thus we see that $\beta_l(v\otimes s)\in \im (\alpha_l)$ if and only if $\beta'_l(v\otimes s)=0$ for all $s\in H^0(C,L-D)$.  We can express this condition using the Serre duality pairing $\langle\cdot,\cdot\rangle$ as follows:
\[ \langle \sigma,\beta'_l(v\otimes s) \rangle=0, \forall \sigma\in H^0(C,K_C-L+D), s\in H^0(C,L-D). \]
Consider now the following commutative square:
 \begin{figure}[H]\centering
  \begin{tikzpicture}
    \matrix (m) [matrix of math nodes,row sep=2em,column sep=1em,minimum width=1em]
  {
     H^0(C,\oo_D(D))\otimes H^0(C,L-D)& H^0(C,L\otimes\oo_D) \\
     H^1(C,\oo_C)\otimes H^0(C,L-D) & H^1(C,L-D)\\};
  \path[-stealth]
    (m-1-1) edge node [auto] {$\beta_l$} (m-1-2)
            edge node [auto]{$\delta\otimes\id$}  (m-2-1)
     (m-1-2) edge node [auto] {$\delta_l$} (m-2-2)
     (m-2-1) edge node [auto] {$\beta''_l$} (m-2-2);
 \end{tikzpicture}
 \end{figure}
\noindent where the map $\beta''_l$ is the cup-product mapping and $\delta:H^0(C,\oo_D(D))\rightarrow H^1(C,\oo_C)$ is the coboundary map with dual given by the restriction map 
\[ \alpha:H^0(C,K_C)\rightarrow H^0(C,K_C\otimes\oo_D). \]
Then we have
\begin{align*}
\langle \sigma,\beta'_l(v\otimes s) \rangle &= \langle \sigma,\beta''_l(\delta v \otimes s) \rangle \\
&=\langle \mu_l(\sigma\otimes s),\delta v \rangle\\
&=\langle \alpha\mu_l(\sigma\otimes s),v \rangle,
\end{align*}
where $\mu_l$ is the cup-product map
\[\mu_l: H^0(C,K_C-L+D) \otimes H^0(C,L-D) \rightarrow H^0(C,K_C). \] 
To sum up, $v\in T_D(V_e^{e-f}(l))$ if and only if $v\in \im (\alpha\mu_l)^{\perp}$, where the superscript $^\perp$ denotes orthogonality with respect to the natural pairing given by the residue between $H^0(C,\oo_D(D))$ and $T_D(C_e)^{\vee}=H^0(C,K_C\otimes \oo_D)$.

Therefore we found that $T_D(V_e^{e-f})(l) = \im (\alpha\mu_l)^{\perp}$, for $l$ complete and $D\notin V^{e-f-1}_e(l)$.

Recall that $\Gamma_e(l)=V_e^r(l)$.  Hence the intersection $\Gamma_e(l_1)\cap\Gamma_e(l_2)$, where $l_1\in G_{d_1}^{r_1}(C)$ and $l_2\in G_{d_2}^{r_2}(C)$ are complete linear series, is transverse at $D=\sum_{i=1}^n a_i p_i \in C_e$, with distinct points $p_i$, if
\[ \im(\alpha\mu_{l_1})^{\perp} + \im(\alpha\mu_{l_2})^{\perp} = T_D C_e, \]
or equivalently
\[ \im(\alpha\mu_{l_1})\cap\im(\alpha\mu_{l_2})=0. \]
In what follows we exhibit some classes of examples for which this transversality condition cannot hold.

\section{Intersections of incidence varieties - proof of Theorem \ref{thm:incidence}}\label{sec:interincid}
In this section we investigate the failure of transversality for intersections of incidence varieties in certain interesting cases and in doing so we prove Theorem \ref{thm:incidence}.

Recall that for two linear series $l_1=\grdop{r_1}{d_1}$ and $l_2=\grdop{r_2}{d_2}$ on a general curve $C$ and for the positive integer $e=r_1+r_2$, we expect there to be a finite number of divisors in the intersection $\Gamma_e(l_1)\cap\Gamma_e(l_2)$ and this number is given by formula (\ref{eq:countincidencegeneral}).

Consider the linear series $l_1=\grdop{r_1}{d_1}$, the pencil $l_2=\grdop{1}{d_2}$, and $e=r_1+1$.  Formula (\ref{eq:countincidencegeneral}) gives that the number of divisors $D\in C_{r_1 +1}$ common to both $l_1$ and $l_2$ is
\begin{equation}\label{eq:countincidence}
(d_1-r_1)\binom{d_2 - 1}{r_1} - g \binom{d_2-2}{r_1-1}.
\end{equation}  
This number was first computed by Severi in the context of the theory of correspondences and coincidences on curves (see Section 74 of \cite{SL}).

From our point of view, this choice of parameters provides an interesting example of a zero count when $d_2=r_1+2$ and $\rho(g,r_1,d_1)=0$, because now
\[ (d_1-r_1)\binom{d_2 - 1}{r_1} - g \binom{d_2-2}{r_1-1} = \rho(g,r_1,d_1)=0. \]
Thus we expect this intersection not to be well-behaved in the case of vanishing $\rho(g,r_1,d_1)$.  

\bigskip

We now clarify when the zero-count comes from an empty intersection or from a degenerate positive-dimensional intersection.  Note that since $\rho(g,r_1,d_1)=0$, it immediately follows that:
\begin{align*}
d_1&=r_1(s_1+1),\\
g&=s_1(r_1+1),
\end{align*}
where $s_1:=g-d_1+r_1$ be the index of speciality of the linear series $l_1$.  Moreover, since the curve $C$ is general, the Brill-Noether number corresponding to the pencil $l_2$ 
\[ \rho(g,1,r_1+2)=s_1(r_1+1)-2(s_1-1)(r_1+1)=(r_1+1)(2-s_1) \]
must be non-negative.  This is only possible if $s_1=1$ or $s_1=2$.  If $s_1=1$, then $l_1=K_C$ and we have the following

\begin{prop}\label{prop:zerocount}
Let $C$ be a general curve of genus $g$, $K_C$ its canonical linear series and $l_2=\grdop{1}{d_2}$.  If $d_2=r_1+2$, then there are two possibilities for the intersection $\Gamma_e(l_1)\cap\Gamma_e(l_2)$:
\begin{enumerate}[label=(\roman*), wide, labelwidth=!, labelindent=0pt]
	\item it is empty if $l_1=K_C$ and $l_2$ is base point free;
	\item it is strictly positive-dimensional if $l_1=K_C$ and $l_2$ is not base point free.
\end{enumerate}
\end{prop}
\begin{proof}
Let $D\in\Gamma_{r_1+1}(l_1)\cap\Gamma_{r_1+1}(l_2)$.  Then $K_C-D\geq 0$ for all $D\in C_{r_1+1}=C_g$ satisfying $g-(r_1+1)+\dim |D|=\dim |D|>0$.  Hence $D\in\Gamma_{g}(K_C)$ if and only if  $|D|=\grdop{1}{g}$.  If $l_2$ is base point free, then the intersection $\Gamma_g(K_C)\cap\Gamma_g(l_2)$ is empty.  Otherwise, the intersection $\Gamma_g(K_C)\cap\Gamma_g(l_2)$ is at least 1-dimensional, hence not a finite, discrete set.
\end{proof}

\bigskip

If $s_1=2$, then $l_1=\grdop{r_1}{3r_1}$ and $l_2=\grdop{1}{r_1+2}$.  In particular, this is the case when $l_1=\grdop{r_1}{3r_1}=K_C-l_2$ and we find ourselves in the situation of Theorem \ref{thm:incidence} i), which we now prove.  
\begin{proof}[of Theorem \ref{thm:incidence} i)]
Assume there exists an effective divisor $D\in\Gamma_{r_1+1}(l_1)\cap\Gamma_{r_1+1}(l_2)$.
Hence there exists an effective divisor $E_1$ of degree $2r_1-1$ such that
\[|D+E_1|=l_1\] and an effective divisor $E_2$ of degree 1 such that
\[ |D+E_2|=l_2. \]
Since the Brill-Noether number of $l_2$ vanishes, $l_2$ is a general linear series and hence base-point free.  It is also complete, as $(2r_1+2)-(r_1+2)+1=r_1+1>0$.  Using the base-point-free pencil trick on the cup product mapping
\[ \mu_0: H^0(C,D+E_2)\otimes H^0(C,K_C-D-E_2)\rightarrow H^0(C, K_C) \]
and the fact that $C$ is a general curve we get that $h^0(K_C-2D-2E_2)=0$.  Combining this with Riemann-Roch, we obtain that $\dim|2D+2E_2|=2$.

Next, $0\leq \dim|E_1|=\dim|l_1-D|=\dim|K_C-l_2-D|=\dim |K_C-2D-E_2|$ which we plug into Riemann-Roch to get that
\[ \dim|2D+E_2| = h^0(K_C-2D-E_2)+1\geq 2. \]
We conclude that $\dim|2D+E_2|=2$.  Moreover, the linear system $|2D+2E_2|$ is base-point free and since $E_2$ is a point, this implies that $\dim |2D+2E_2|=3$, which is a contradiction.  Hence the intersection $\Gamma_{r_1+1}(l_1)\cap\Gamma_{r_1+1}(l_2)$ is empty in this case.
\end{proof}

\section{Degenerations of secant varieties}\label{sec:degenerationssecant}
In this section we construct a space of degenerations of secant varieties for families of curves of compact type using limit linear series and the idea of degeneracy loci.  

Before doing so, we recall some well-known facts about limit linear series.  Consider a family $\pi:\mathscr{X}\rightarrow B$ of curves of genus $g$ together with a section $\sigma:B\rightarrow\mathscr{X}$ such that $B=\text{Spec}(R)$ for some discrete valuation ring $R$ with uniformising parameter $t$.
Assume further that $\mathscr{X}$ is a nonsingular surface, projective over $B$.  Let $0\in B$ denote the point corresponding to the maximal ideal of $R$ and $\eta$, $\bar{\eta}$ the generic and geometric generic point of $B$, respectively.  Moreover, let the special fibre $\mathscr{X}_0$ be a reduced curve of compact type, while $\mathscr{X}_{\bar{\eta}}$ is assumed to be a smooth, irreducible curve of the same genus.  
 
Let $(\mathscr{L}_{\bar{\eta}},\mathscr{V}_{\bar{\eta}})$ be a $\grd$ on $\mathscr{X}_{\bar{\eta}}$.  We now briefly explain how this gives rise to a limit series on $\mathscr{X}_0$, after possibly replacing nodes of $\mathscr{X}_0$ by smooth rational curves via base change.  For details, see \cite{EH86}.  After the base change we may assume that $(\mathscr{L}_{\bar{\eta}},\mathscr{V}_{\bar{\eta}})$ comes from a series $(\mathscr{L}_{\eta},\mathscr{V}_{\eta})$ of type $\grd$ on $\mathscr{X}_{\eta}$.  This then determines a $\grd$ on each irreducible component $Y$ of $\mathscr{X}_0$ as follows: since $\mathscr{X}$ is smooth, $\mathscr{L}_{\eta}$ extends to a line bundle on $\mathscr{X}$ unique up to tensoring with a line bundle of the form $\oo_{\mathscr{X}}(\mathcal{C})$, where $\mathcal{C}$ is supported on $\mathscr{X}_0$.  There exists therefore an extension $\mathscr{L}_Y$ to $\mathscr{L}_{\eta}$, unique up to isomorphism, such that $\deg(\mathscr{L}_Y|_Y)=d$ and for any irreducible component $Z$ of $\mathscr{X}_0$ with $Z\neq Y$, $\deg(\mathscr{L}_Y|_{Z})=0$.  We set $\mathscr{V}_Y=(\mathscr{V}_{\eta}\cap\pi_{*}\mathscr{L}_Y)\otimes k(0)$ and it follows that
 \[ \mathscr{V}_Y \simeq \pi_{*}\mathscr{L}_Y\otimes k(0)\subseteq H^0(\mathscr{L}_Y|_{\mathscr{X}_0}) \]
is a vector space of dimension $r+1$ which we will moreover identify with its image in $H^0(\mathscr{L}_Y|_Y)$.  Hence the pair $(\mathscr{L}_Y|_Y,\mathscr{V}_Y)$ is a $\grd$ on $Y$, which is called the $Y$-\textit{aspect} of $(\mathscr{L}_{\eta},\mathscr{V}_{\eta})$.  The collection of aspects
\[ l = \{ (\mathscr{L}_Y|_Y,\mathscr{V}_Y) \mid Y \text{ component of }\mathscr{X}_0 \} \]
is called the \textit{limit} of $(\mathscr{L}_{\eta},\mathscr{V}_{\eta})$.

Now, for a curve $X$ of compact type a  \textit{crude limit linear series} $l$ is 
\[ l:=\{ l_Y \text{ a }\grd\text{ on }Y\mid Y \text{ component of }X\} \]
together with a \textit{compatibility condition} on the \textit{vanishing sequence} at the point $p\in Y$
\[ 0 \leq a_0(l_Y,p) < a_1(l_Y,p) < \cdots < a_r(l_Y,p)\leq d, \]
where the $a_i(l_Y,p)$ are the orders with which non-zero sections of $l_Y$ vanish at $p$.   The compatibility condition is: if $Z$ is another component of $\mathscr{X}_0$ with $Y\cap Z = p$, then for all $i=0,\ldots,r$,
\begin{equation}\label{eq:vanish}
 a_i(l_Y,p) + a_{r-i}(l_Z,p) \geq d. 
\end{equation}
It follows (for a proof, see \cite{EH86}) that the limit of $(\mathscr{L}_{\eta},\mathscr{V}_{\eta})$ from above is a crude limit linear series.  Note that if we have equality in (\ref{eq:vanish}), then we have a \textit{refined limit linear series}.  In general we omit the adjectives ``crude'' or ``refined'' unless necessary.

Recall also the definition of the \textit{ramification sequence} at the point $p\in Y$:
\[ 0 \leq \alpha_0(l_Y,p) < \alpha_1(l_Y,p) < \cdots < \alpha_r(l_Y,p)\leq d-r,  \]
where $\alpha_i(l_Y,p)=a_i(l_Y,p)-i$.

The Plücker formula for refined limit linear series (cf.~Proposition 1.1 of \cite{EH86}) states the following: if $X$ is a genus $g$ curve of compact type and $l$ is a limit linear series of type $\grd$ on $X$, then
	\begin{equation}\label{thm:pluckerlls}
	\sum_{q\text{ smooth point of }X}\biggl(\sum_{i=0}^r \alpha^l_i(q) \biggr)=(r+1)d+\binom{r+1}{2}(2g-2).
	\end{equation}

An alternative description for limit linear series has been developed by Osserman in a series of papers starting with \cite{Os} and we use it in our construction of degenerations of secant varieties to a family of nodal curves of compact type.  We now give a summary of the most important facts and definitions in this approach.  We base our discussion on the papers \cite{Os2} and \cite{Os3} to which we will frequently refer for details.

\textbf{Case 0.} We start with the description of limit linear series on a single nodal curve $X$ of compact type. 
Let $\Gamma$ be the dual graph of $X$ and denote by $V(\Gamma)$ its set of vertices and by $E(\Gamma)$ its set of edges.  For a vertex $v\in V(\Gamma)$, let $Y^v$ be the corresponding irreducible component of $X$ and $Y^v_c$ the closure of the complement of $Y^v$ in $X$.  We call an \textit{enriched structure} (cf.~\cite[Definition 2.14]{Os2}) a collection $\{\oo_v\}_{v\in V(\Gamma)}$ of line bundles on $X$ such that $\bigotimes_{v\in V(\Gamma)}\oo_v\simeq \oo_X$ and, for any $v\in V(\Gamma)$:
\[ \oo_v|_{Y^v}\simeq \oo_{Y^v}(-(Y^v\cap Y^v_c)) \text{ and }\oo_v|_{Y^v_c}\simeq\oo_{Y^v_c}(Y^v\cap Y^v_c).  \]
Such structures always exist and since $X$ is of compact type they are unique.  Now let $L$ be a line bundle of degree $d$ on $X$.  A \textit{multidegree} $\vec{d}=(d_v)_{v\in V(\Gamma)}$ of $d$ is a vector with integer entries $d_v$ satisfying $\sum_{v\in V(\Gamma)}d_v=d$. Set $\vec{d}^v:=(0,\ldots,0,d,0,\ldots,0)$ with entry $d$ at $v$ and 0 elsewhere.  We say that $L$ has multidegree $\vec{d}$ if $\deg L|_{Y^v}=d_v$ for all $v\in V(\Gamma)$. Note that multiplying $L$ with $\oo_v$ alters the multidegree in the following way: $\deg (L\otimes\oo_v)|_{Y^v}$ decreases by the number of vertices adjacent to $v$ in $\Gamma$; if $v'\neq v$ is an adjacent vertex to $v$, then $\deg (L\otimes\oo_v)|_{Y^{v'}}$
increases by 1; if $v'\neq v$ is not adjacent to $v$, then $\deg (L\otimes\oo_v)|_{Y^{v'}}$ remains unchanged.  We call the bundles $\oo_v$ \textit{twisting bundles} of $L$ and we say that $L\otimes\oo_v$ is a \textit{twist of} $L$ \textit{at} $v$.  

We may reorganise the twisting bundles of $L$ by considering instead a collection $\{\oo_{(e,v)}\}_{v\in V(\Gamma),e\in E(\Gamma)}$ of line bundles on $X$, where $v$ is a vertex adjacent to the edge $e$, and satisfying the following: given an edge $e\in E(\Gamma)$ and $v\neq v'$ its two adjacent vertices we have $\oo_{(e,v)}|_{Y_{(e,v)}}\simeq\oo_{Y_{(e,v)}}(-(Y^v\cap Y^{v'}))$ and $\oo_{(e,v)}|_{Y_{(e,v')}}\simeq\oo_{Y_{(e,v)}}(Y^v\cap Y^{v'})$, where $Y_{(e,v)}$ denotes the union of the irreducible components of $X$ whose vertices lie in the same connected component of $\Gamma\setminus e$ as $v$.  We have that $X=Y_{(e,v)}\cup Y_{(e,v')}$, so that line bundles on $X$ are uniquely determined by their restrictions to $Y_{(e,v)}$ and $Y_{(e,v')}$.  We also see that $\oo_{(e,v)}\otimes\oo_{(e,v')}\simeq \oo_X$, and furthermore one can check that $\oo_v = \oo_{(e_1,v)}\otimes\cdots\otimes\oo_{(e_n,v)}$, where $e_1,\ldots,e_n$ are all the edges adjacent to $v$.  Note that there is a canonical section $\oo_X\rightarrow\oo_{(e,v)}$ determined as the zero map on $Y_{(e,v)}$ and the canonical inclusion on $Y_{(e,v')}$.  Moreover, the isomorphism $\oo_{(e,v)}\otimes\oo_{(e,v')}\simeq \oo_X$ together with the above section induce a map $\oo_{(e,v')}\rightarrow\oo_X$.  Analogously one obtains a canonical section $\oo_X\rightarrow\oo_{(e,v')}$ and a map $\oo_{(e,v)}\rightarrow\oo_X$.

Given a line bundle $L$ of degree $d$ on $X$ of multidegree $\vec{d}$ and another multidegree  $\vec{d'}$ of $d$ one can always find a unique minimal sequence of twists of $L$ such that the resulting twisted bundle has multidegree $\vec{d'}$.  Denote the product of line bundles $\oo_v$ (or $\oo_{(e,v)}$) occurring in the aforementioned minimal sequence of twists by $\oo_{\vec{d},\vec{d'}}$.

Now let $L$ be a line bundle of degree $d$ and fixed multidegree $\vec{d_0}$ on $X$. 
 For any other multidegree $\vec{d}\neq\vec{d_0}$ of $d$, set  $L^{\vec{d}}:=L\otimes\oo_{\vec{d^0},\vec{d}}$ and $L^v:=L^{\vec{d^v}}|_{Y^v}$.
Given two multidegrees $\vec{d}$, $\vec{d'}$ of $d$, the twists $\oo_{\vec{d},\vec{d'}}$ induce 
a morphism $f_{\vec{d},\vec{d'}}:L^{\vec{d}}\rightarrow L^{\vec{d'}}$.  We now describe how such a morphism is obtained (see \cite[Notation 2.19]{Os2} for more details and for treatment also in the case of nodal curves not of compact type).  Assume first that $\vec{d}$ and $\vec{d'}$ are such that 
$L^{\vec{d}}=L^{\vec{d'}}\otimes \oo_{(e,v)}$ or $L^{\vec{d'}}=L^{\vec{d}}\otimes \oo_{(e,v')}$ for some edge $e\in E(\Gamma)$ and its two adjacent vertices $v\neq v'$.  We define a map $f_e:L^{\vec{d}}\rightarrow L^{\vec{d'}}$ as follows: in the first case $f_e$ is induced by the map  $\oo_{(e,v)}\rightarrow\oo_X$ and in the second case $f_e$ is induced by the section $\oo_X\rightarrow\oo_{(e,v')}$.   If $\vec{d}$ and $\vec{d'}$ are such that $\oo_{\vec{d},\vec{d'}}$ consists of a product of several twist bundles $\oo_{(e_i,v_i)}$, then we define $f_{\vec{d},\vec{d'}}:L^{\vec{d}}\rightarrow L^{\vec{d'}}$ as the composition of the maps $f_{e_i}$ as above.  We also obtain an induced morphism $H^0(X,L^{\vec{d}})\rightarrow H^0(X,L^{\vec{d'}})$. 

A limit linear series of type $\grd$ on the nodal curve of compact type $X$ with enriched structure $\{\oo_v\}_{v\in V(\Gamma)}$ and fixed multidegree $\vec{d_0}$ of $d$ consists of a tuple $(L,(V^v)_{v\in V(\Gamma)})$ where $L$ is a  line bundle of degree $d$ and multidegree $\vec{d_0}$ and the $V^v$ are subspaces of $H^0(X,L^v)$, such that for all multidegrees $\vec{d}$ of $d$, the natural morphism
\[H^0(X,L^{\vec{d}})\rightarrow \bigoplus_{v\in V(\Gamma)}H^0(Y^v,L^v)/V^v,\]
 induced by the maps $f_{\vec{d},\vec{d^v}}$ and restrictions to irreducible components has kernel of dimension at least $r+1$ (see \cite[Definition 2.21]{Os2} for details).  
 
We now summarise the construction of the moduli scheme of limit linear series for families of curves.  From now on, let $B$ be a scheme and $f:T\rightarrow B$ a $B$-scheme.

\textbf{Case 1.}  First, let $\mathscr{X}\rightarrow B$ be a smooth proper family of smooth curves of fixed genus equipped with a section. 
 The functor $\mathscr{G}^r_d(\mathscr{X}/B)$ of linear series of type $\grd$ is defined by associating to each $B$-scheme $T$ the set of equivalence classes of pairs $(\mathscr{L},\mathscr{V})$, where $\mathscr{L}$ is a line bundle of relative degree $d$ on $\mathscr{X}\times_B T$ and $\mathscr{V}\subseteq	\pi_{2*}\mathscr{L}$ is a subbundle of rank $r+1$, where $\pi_2:\mathscr{X}\times_B T\rightarrow T$ is the usual projection.  We say that the pairs $(\mathscr{L},\mathscr{V})$ and $(\mathscr{L}',\mathscr{V}')$ are equivalent if there exists a line bundle $\mathscr{M}$ on $T$ and an isomorphism $\varphi:\mathscr{L}\rightarrow\mathscr{L}'\otimes \pi_2^*\mathscr{M}$ such that $\pi_{2*}\varphi$ maps $\mathscr{V}$ into $\mathscr{V}'$.  The last condition makes sense because of the following: by the projection formula, there is a natural isomorphism $\pi_{2*}(\mathscr{L}'\otimes\pi_2^*\mathscr{M})\simeq(\pi_{2*}\mathscr{L}')\otimes\mathscr{M}$, which means that there is an induced morphism $\pi_{2*}\mathscr{L}\rightarrow(\pi_{2*}\mathscr{L}')\otimes\mathscr{M}$.
Take an open cover $\{U_i\}$ of $T$ such that $\mathscr{M}|_{U_i}\simeq\oo_{U_i}$.  Thus, for each $U_i$ we have induced isomorphisms $\pi_{2*}\mathscr{L}|_{U_i}\rightarrow(\pi_{2*}\mathscr{L}'|_{U_i})\otimes\oo_{U_i}=\pi_{2*}\mathscr{L}'|_{U_i}$.  The condition requires that the above morphism restricts to a morphism $\mathscr{V}|_{U_i}\rightarrow\mathscr{V}'|_{U_i}$ for all $U_i$.  One can see that if the condition is satisfied by a trivialising cover $\{U_i\}_{i\in I}$, then it will be satisfied by any other such cover $\{\widetilde{U}_j\}_{j\in J}$ because the transition functions of $\mathscr{M}$ on $U_i\cap \widetilde{U}_j$ are elements of $\oo^*_{U_i\cap \widetilde{U}_j}$ which leave the $\mathscr{V}'|_{U_i\cap \widetilde{U}_j}$ unchanged.
  The functor we just described is represented by a scheme $G^r_d(\mathscr{X}/B)$ which is proper over $B$.

\textbf{Case 2.}  Next, let $\mathscr{X}\rightarrow B$ be a flat proper family of nodal curves of compact type of fixed genus such that no nodes are smoothed and where $B$ is regular and connected.    For details of the constructions and results we refer to \cite{Os3}.
All fibres will have the same dual graph $\Gamma$ and for each vertex $v\in V(\Gamma)$, we denote the corresponding irreducible component of $\mathscr{X}$ by $\mathscr{Y}^v$.  For a line bundle $\mathscr{L}$ of relative degree $d$ on $\mathscr{X}\times_B T$, we say that it has multidegree $\vec{d}=(d_v)_{v\in V(\Gamma)}$ if $\mathscr{L}|_{\mathscr{Y}^v\times_B T}$ is of relative degree $d_v$, for all $v\in V(\Gamma)$.  
We also have an enriched structure $\{\oo_v\}_{v\in V(\Gamma)}$ on $\mathscr{X}$ defined as in the case of a single curve, i.e.~the line bundles $\oo_v$ on $\mathscr{X}$ satisfy $\bigotimes_{v\in V(\Gamma)}\oo_v\simeq\oo_{\mathscr{X}}$ and for all $v\in V(\Gamma)$:
\[ \oo_v|_{\mathscr{Y}^v}\simeq \oo_{\mathscr{Y}^v}(-(\mathscr{Y}^v\cap \mathscr{Y}^v_c)) \text{ and }\oo_v|_{\mathscr{Y}^v_c}\simeq\oo_{\mathscr{Y}^v_c}(\mathscr{Y}^v\cap \mathscr{Y}^v_c),  \]
where $\mathscr{Y}^v\cap \mathscr{Y}^v_c$ should be interpreted as the corresponding divisor on $\mathscr{Y}^v$.  As before, we also have the alternative formulation of the twist bundles in terms of the $\oo_{(e,v)}$, defined as in the case of the single nodal curve (replacing of course $Y$ by $\mathscr{Y}$).

 Now fix a choice of multidegree $\vec{d_0}$ of $d$.  Given a line bundle $\mathscr{L}$ of relative degree $d$ on $\mathscr{X}\times_B T$ and multidegree $\vec{d_0}$ and another multidegree $\vec{d}\neq\vec{d_0}$ of $d$, we construct a line bundle $\mathscr{L}^{\vec{d}}$ on $\mathscr{X}\times_B T$ of multidegree $\vec{d}$ as in the case of single curves via a minimal sequence of twists of $\mathscr{L}$ by bundles $\pi_1^*\oo_v$, where $\pi_1:\mathscr{X}\times_B T\rightarrow \mathscr{X}$ is the projection.  As before, given two multidegrees $\vec{d}$ and $\vec{d'}$ of $d$, let $f_{\vec{d},\vec{d'}}:\mathscr{L}^{\vec{d}}\rightarrow\mathscr{L}^{\vec{d'}}$ be the unique map obtained by performing the minimal number of twists. 

The Picard functor $\mathscr{P}^{\vec{d_0}}(\mathscr{X}/B)$ is defined by associating to each $B$-scheme $T$ the set of isomorphism classes of line bundles $\mathscr{L}$ of relative multidegree $\vec{d_0}$ on $\mathscr{X}\times_B T$.  Similarly, $\mathscr{P}^d(\mathscr{Y}^v/B)$ denotes the Picard functor of line bundles of relative degree $d$. We now define $\mathscr{P}$ as the following fibre product:
 \begin{figure}[H]\centering
  \begin{tikzpicture}
    \matrix (m) [matrix of math nodes,row sep=4em,column sep=7em,minimum width=1em]
  {
     \mathscr{P} &   \prod_{v\in V(\Gamma)}\mathscr{G}^r_{d}(\mathscr{Y}^v/B) \\
     \mathscr{P}^{\vec{d_0}}(\mathscr{X}/B) & \prod_{v\in V(\Gamma)}\mathscr{P}^d(\mathscr{Y}^v/B)\\};
  \path[-stealth]
    (m-1-1) edge  (m-1-2)
            edge  (m-2-1)
     (m-1-2) edge node [auto]{$\prod_{v\in V(\Gamma)}\phi_v$} (m-2-2)
     (m-2-1) edge node [auto]{$(\psi_v)_{v\in V(\Gamma)}$} (m-2-2);
 \end{tikzpicture}
 \end{figure}

\noindent where for all $v\in V(\Gamma)$, $\phi_v:\mathscr{G}^r_d(\mathscr{Y}^v/B)\rightarrow\mathscr{P}^d(\mathscr{Y}^v/B)$ is the forgetful morphism $(\mathscr{M},\mathscr{V})\mapsto \mathscr{M}$ and $\psi_v:\mathscr{P}^{\vec{d_0}}\rightarrow\mathscr{P}^d(\mathscr{Y}/B)$ maps $\mathscr{L}$ to $\mathscr{L}^v=\mathscr{L}^{\vec{d}^v}|_{\mathscr{Y}^v\times_B T}$ via twists and restriction to $\mathscr{Y}^v\times_B T$.  

A $T$-valued point of $\mathscr{P}$ consists of equivalence classes (equivalence relation as in Case 1) of tuples $(\mathscr{L},(\mathscr{V}^v)_{v\in V(\Gamma)})$, where $\mathscr{L}$ is a line bundle of relative multidegree $\vec{d_0}$ on $\mathscr{X}\times_B T$ and the $\mathscr{V}^v$ are rank-$(r+1)$ subbundles of $\pi_{2*}\mathscr{L}^v$.

We define the functor $\mathscr{G}^r_d(\mathscr{X}/B)$ as a subfunctor of the functor of points of $\mathscr{P}$ as follows: a tuple $(\mathscr{L},(\mathscr{V}^v)_{v\in V(\Gamma)})$ as above is a $T$-valued point of $\mathscr{G}^r_d(\mathscr{X}/B)(T)$ if for all multidegrees $\vec{d}$ of $d$, the map
\begin{equation}\label{eq:condlls1}
\pi_{2*}\mathscr{L}^{\vec{d}}\rightarrow\bigoplus_v(\pi_{2*}\mathscr{L}^v)/\mathscr{V}^v 
\end{equation} 
induced by the restriction to the component $\mathscr{Y}^v$ and the twist maps $f_{\vec{d},\vec{d}^v}$ has its $(r+1)$-st degeneracy locus equal to the whole of $T$.   The functor we constructed is also represented by a scheme $G^r_d(\mathscr{X}/B)$ proper over $B$ and can be shown to be independent (up to isomorphism) of the choice of fixed multidegree $\vec{d_0}$.

\textbf{Case 3.}  Now assume that $\pi:\mathscr{X}\rightarrow B$ is a family of curves of compact type with nodes that may be smoothed and satisfying the conditions listed in Definition 3.1 of \cite{Os}.  We call such a family a \textit{smoothing family}.  In this case the dual graph of the fibres may vary and the components $\mathscr{Y}^v$ may not be defined over all of $B$.  Denote by $\Gamma_b$ the dual graph of the fibre $\mathscr{X}_b$ for some $b\in B$.  One can check that if $b$ specialising to $b'$ are points of $B$, then there is a unique contraction map $cl_{b,b'}:\Gamma_{b'}\rightarrow\Gamma_{b}$ induced on vertices by associating to a component $Y'$ of $\mathscr{X}_{b'}$ the component $Y$ of $\mathscr{X}_b$ containing $Y'$ in its closure.  On edges, $cl_{b,b'}$ maps $e\in E(\Gamma_{b'})$ to the corresponding edge of $\Gamma_b$ if there is a node of $\mathscr{X}_b$ specialising to the node on $\mathscr{X}_{b'}$ corresponding to $e$; otherwise, $e$ is contracted.  If $b'$ specialises to $b''$, then $cl_{b,b''}=cl_{b,b'}\circ cl_{b',b''}$. 
Assume also that there exists a unique maximally degenerate fibre over some $b_0\in B$ with dual graph $\Gamma_0$ and that, for each $b\in B$, there exist contractions $cl_b:\Gamma_0\rightarrow\Gamma_b$ (unique up to automorphism of $\Gamma_0$) such that if $b$ specialises to $b'$, then $cl_b=cl_{b,b'}\circ cl_{b'}$.  For details see \cite[Section 2]{Os3}.  With this additional assumption the family is called an \textit{almost local smoothing family}.

 If a node corresponding to $e$ is not smoothed by $\pi$ and occurs in the fibre $\mathscr{X}_b$, we call its corresponding edge $e_b$.   
To make sense of the the twist bundles $\oo_{(e,v)}$, we need to first introduce some notation (for details we again refer to \cite[Section 2]{Os3}).  Let $\Delta'$ denote the non-smooth locus of $\pi$ and, for $e\in E(\Gamma_0)$, let $\Delta'_e$ denote the connected component of $\Delta'$ consisting of the union of non-smoothed nodes corresponding to edges $e_b$ of the dual graphs of fibres $\mathscr{X}_b$. 
Let $\Delta_e\subset B$ denote the image of $\Delta'_e$ under $\pi$ and, given a $v\in V(\Gamma_0)$ adjacent to $e$, let $\mathscr{Y}_{(e,v)}\in\pi^{-1}(\Delta_e)$ be the unique closed set such that for each $b\in\Delta_e$, the fibre $(\mathscr{Y}_{(e,v)})_b$ is equal to the union of the components of $\mathscr{X}_b$ corresponding to the vertices of $\Gamma_b$ lying in the same connected component as $v$ in $\Gamma_b\setminus e_b$.
Furthermore, one can show that $\Delta_e$ is a divisor of $B$ and if $v\neq v'$ are two vertices adjacent to an edge $e\in\Gamma_0$, then $\mathscr{Y}_{(e,v)}\cup\mathscr{Y}_{(e,v')}=\pi^{-1}\Delta_e$ and $\mathscr{Y}_{(e,v)}\cap\mathscr{Y}_{(e,v')}=\Delta'_e$.  In this case $\mathscr{Y}_{(e,v)}$ is a divisor in $\mathscr{X}$ and we set $\oo_{(e,v)}=\oo_{\mathscr{X}}(\mathscr{Y}_{(e,v)})$.
Moreover we now have that the $\oo_{(e,v)}\otimes\oo_{(e,v')}$ are isomorphic to $\oo_{\mathscr{X}}$ only locally on $B$: $\oo_{(e,v)}\otimes\oo_{(e,v')}=\oo_{\mathscr{X}}(\pi^{-1}\Delta_e)=\pi^*\oo_B(\Delta_e)$, so for an open cover $\{U_i\}_{i\in I}$ of $B$ trivialising $\pi^*\oo_B(\Delta_e)$ we have that 
\begin{equation}\label{eq:trivialtensor}
(\oo_{(e,v)}\otimes\oo_{(e,v')})|_{\pi^{-1}(U_i)}\simeq \oo_{\pi^{-1}(U_i)}.
\end{equation}
From now on we fix such a cover $\{U_i\}_{i\in I}$ of $B$.
 
 Let $\mathscr{L}$ be a line bundle on $\mathscr{X}\times_B T$ of relative degree $d$ and fixed multidegree $\vec{d_0}$.  This means that its restriction to $\mathscr{X}_{b_0}\times_B T$ has multidegree $\vec{d_0}$, while the restrictions to the other $\mathscr{X}_b\times_B T$ have the unique multidegree $\vec{d_0^b}$ resulting as follows: if the edge $e\in E(\Gamma_0)$ is contracted, then we replace the vertices $v\neq v'$ adjacent to $e$ by a vertex $w$ and set $(d_0^b)_w=(d_0)_v + (d_0)_{v'}$; if $e$ is not contracted, then $(d_0^b)_v=(d_0)_v$ and $(d_0^b)_{v'}=(d_0)_{v'}$ for the adjacent vertices $v\neq v'$ to $e$.  Given another multidegree $\vec{d}\neq\vec{d_0}$ of $d$, we obtain a line bundle $\mathscr{L}^{\vec{d}}$ on $\mathscr{X}\times_B T$ of multidegree $\vec{d}$ from $\mathscr{L}$ by multiplying with twisting bundles $\pi_1^*\oo_{(e,v)}$ as in Case 2.   
  Given any two multidegrees $\vec{d}$ and $\vec{d'}$ of $d$ and $i\in I$, we have maps
 \[ f_{\vec{d},\vec{d'},i}:\mathscr{L}^{\vec{d}}|_{\pi^{-1}(U_i)}\rightarrow\mathscr{L}^{\vec{d'}}|_{\pi^{-1}(U_i)}, \]
induced in a similar way as in Case 2, but this time using the local isomorphisms (\ref{eq:trivialtensor}).
 
To define the functor $\mathscr{G}^r_d(\mathscr{X}/B)$ in this case we consider the functor that associates to $T$ the set of equivalence classes of tuples $(\mathscr{L}, (\mathscr{V}^v)_{v\in V(\Gamma_0)})$, where $\mathscr{L}$ is a line bundle of fixed multidegree $\vec{d_0}$ on $\mathscr{X}\times_B T$ and for each vertex $v\in V(\Gamma_0)$, the $\mathscr{V}^v$ are subbundles of rank $r+1$ of the $\pi_{2*}\mathscr{L}^{\vec{d}^v}$, where we now use the bundles $\mathscr{L}^{\vec{d^v}}$ which may be defined over the whole family instead of the bundles $\mathscr{L}^v$ which required the existence of the components $\mathscr{Y}^v$.  Set $f:T\rightarrow B$.  We say that a $T$-valued point $(\mathscr{L}, (\mathscr{V}^v)_{v\in V(\Gamma_0)})$ is in $\mathscr{G}^r_d(\mathscr{X}/B)(T)$ if for all $i\in I$ and all multidegrees $\vec{d}$ of $d$, the map
\begin{equation}\label{eq:condlls2}
\pi_{2*}\mathscr{L}^{\vec{d}}|_{(f\circ\pi_2)^{-1}(U_i)}\rightarrow
\bigoplus_v \Bigl(\pi_{2*}\mathscr{L}^{\vec{d}^v}|_{(f\circ\pi_2)^{-1}(U_i)}\Bigr)/\mathscr{V}^v|_{f^{-1}(U_i)}, 
\end{equation} 
induced by the local versions of the twist maps, has its $(r+1)$-st degeneracy locus equal to all of $U_i$.  It can be shown that the functor described above is independent (up to isomorphism) of choice of open cover $\{U_i\}_{i\in I}$ and fixed multidegree $\vec{d_0}$.

\textbf{Relation between Case 2 and Case 3.} Note that Case 2 can be seen as a special case of Case 3, i.e.~with $\Delta_e=B$ for all edges $e\in E(\Gamma_0)$.  One can show that in this case the two constructions yield isomorphic moduli functors of limit linear series and the natural isomorphism between them is induced by restriction to the components $\mathscr{Y}^v$ of $\mathscr{X}$. 

All the constructions above are compatible with base change and the fibre over $t\in B$ is a space of Eisenbud-Harris limit linear series when $\mathscr{X}_t$ is reducible and a space of usual linear series when $\mathscr{X}_t$ is smooth.

We now describe a scheme that parametrises secant varieties for a family of nodal curves of compact type equipped with limit linear series.
\begin{prop}\label{prop:secspacecomtype}
Fix a family of curves $\mathscr{X}\rightarrow B$ over a scheme $B$ like in Case 1, Case 2, or Case 3 above and equipped with a linear series $\ell$ of type $\grd$.
There exists a scheme $\mathcal{V}^{e-f}_{e}(\mathscr{X},\ell)$ proper over $B$, compatible with base change, whose point over every $t\in B$ parametrises pairs $[\mathscr{X}_t,\mathscr{D}_t]$ of curves and divisors such that $\mathscr{D}_t$ is an $(e-f)$-th secant divisor of $\ell_t$.  Furthermore, every irreducible component of $\mathcal{V}^{e-f}_{e}(\mathscr{X},\ell)$ has dimension at least $\dim B +e- f(r+1-e+f)$.
\end{prop}
\begin{proof}
 We construct the functor $\mathcal{V}^{e-f}_{e}(\mathscr{X},\ell)$ as a subfunctor of the functor of points of the symmetric product $Sym^e(\mathscr{X}/B)$ (which we also denote by $Sym^e(\mathscr{X}/B)$).  From the degeneracy locus construction it follows that it is representable by a scheme that is proper over $B$ and compatible with base change and which we also denote by $\mathcal{V}^{e-f}_{e}(\mathscr{X},\ell)$. 
 
\textbf{Case 1.} Suppose first that the family $\mathscr{X}\rightarrow B$ is like in Case 1 and all the fibres of the family are nonsingular.  In this case, for any $T\rightarrow B$ scheme over $B$, we saw above that $\ell=\grd$ on $\mathscr{X}/B$ is given by a pair $(\mathscr{L},\mathscr{V})$, where $\mathscr{L}$ is a line bundle of degree $d$ on $\mathscr{X}\times_B T$ and $\mathscr{V}\subseteq\pi_{2*}\mathscr{L}$ is a vector bundle of rank $r+1$ on $B$, where $\pi_2$ is the second projection from the fibre product onto $T$.  Let $\mathcal{U}\subset\mathscr{X}\times_B Sym^e(\mathscr{X}/B)$ denote the universal family and $\mathcal{U}_T=\mathcal{U}\times_B T$.  Consider the following diagram 
 \begin{figure}[H]\centering\label{picture}
  \begin{tikzpicture}
    \matrix (m) [matrix of math nodes,row sep=2em,column sep=1em,minimum width=1em]
  {
     & (\mathscr{X}\times_B T)\times_B Sym^e(\mathscr{X}/B) & \supset \mathcal{U}_T \\
     \mathscr{X}\times_B T & & T\\};
  \path[-stealth]
    (m-1-2) edge node [auto,swap] {$\tau_1$} (m-2-1)
            edge node [auto]{$\tau_2$}  (m-2-3);
 \end{tikzpicture}
 \end{figure}
\noindent where $\tau_1$ and $\tau_2$ are the usual projections.  
Then the $T$-valued point $[\mathscr{X}\times_B T,\mathscr{D}]\in Sym^e(\mathscr{X}/B)(T)$ belongs to $\mathscr{V}^{e-f}_{e}(\mathscr{X},\ell)(T)$ if the $(e-f)$-th degeneracy locus of the map
\[ \mathscr{V} \rightarrow (\tau_2)_{*}(\tau_1^*\mathscr{L}\otimes\oo_{\mathcal{U}_T}) \]
is the whole of $T$.  By construction $\mathscr{V}^{e-f}_{e}(\mathscr{X},\ell)$ is compatible with base change, so it is a functor, and it has the structure of a closed subscheme, hence it is representable and the associated scheme is proper.  

\textbf{Case 2.}  Now suppose that we are in Case 2 above and the fibres have nodes that are not smoothed by the family $\mathscr{X}$ and as usual let $\Gamma$ denote the dual graph of the fibres.  As we saw above, in this case $\ell=\grd$ on $\mathscr{X}$ is, for any $T\rightarrow B$ scheme over $B$, a tuple $(\mathscr{L},(\mathscr{V}^v)_{v\in V(\Gamma)})$, with $\mathscr{L}$ a line bundle of fixed multidegree $\vec{d_0}$ of	 $d$ on $\mathscr{X}\times_B T$ and $\mathscr{V}^v\subset (\pi_2)_* \mathscr{L}^v$ a subbundle of rank $r+1$ on $T$ subject to the condition on the map in (\ref{eq:condlls1}).  We define $\mathcal{V}'(\mathscr{Y}^v,\mathscr{V}^v)\subset Sym^e(\mathscr{Y}^v/B)$ as follows: we say that a $T$-valued point $[\mathscr{Y}^v\times_B T,\mathscr{E}^v]\in Sym^e(\mathscr{Y}^v/B)(T)$ belongs to $\mathcal{V}'(\mathscr{Y}^v,\mathscr{V}^v)(T)$ if the $(e-f)$-th degeneracy locus of the map 
\[ \mathscr{V}^v \rightarrow (\tau_2)_{*}(\tau_1^*\mathscr{L}^v\otimes\oo_{\mathcal{U}_T}) \]
is the whole of $T$, where the $\tau_1,\tau_2$, and $\mathcal{U}_T$ are defined analogously for  the family $\mathscr{Y}^v$.

Thus, a $T$-valued point $[\mathscr{X}\times_B T,\mathscr{D}]\in Sym^e(\mathscr{X}/B)(T)$ belongs to $\mathscr{V}^{e-f}_{e}(\mathscr{X},\ell)(T)$ if, for all vertices $v$ of $\Gamma$, the $T$-valued points $[\mathscr{Y}^v\times_B T,\mathscr{D}^v+q_i^v]$ belong to $\mathcal{V}'(\mathscr{Y}^v,\mathscr{V}^v)(T)$, where $\mathscr{D}^v$ now denotes the specialisation of the relative divisor $\mathscr{D}$ on the component $\mathscr{Y}^v\times_B T$ and $q_i^v$ denote the preimages of the nodes belonging to $\mathscr{Y}^v\times_B T$ (that may appear with multiplicity so that the relative divisor $\mathscr{D}^v + q_i^v$ is of correct degree $e$).  While this is the most useful description for practical applications, representability is best seen by treating Case 2 as a special case of Case 3 as explained in the summary above.

\textbf{Case 3.}  Now suppose the family $\mathscr{X}$ is a local smoothing family as in Case 3 above.  As we have seen already, a limit linear series $\ell$ of type $\grd$ on $\mathscr{X}$ is, for any $T\rightarrow B$ scheme over $B$, a tuple $(\mathscr{L},(\mathscr{V}^v)_{v\in V(\Gamma_0)})$, where $\Gamma_0$ is the dual graph of the unique maximally degenerate fibre $\mathscr{X}_{b_0}$ of the family, $\mathscr{L}$ is a line bundle of fixed multidegree $\vec{d_0}$ on $\mathscr{X}\times_B T$, and for each $v\in V(\Gamma_0)$, the $\mathscr{V}^v$ are subbundles of rank $r+1$ of the twists $\pi_{2*}\mathscr{L}^{\vec{d}^v}$, subject to condition on the maps in (\ref{eq:condlls2}). 

As we mentioned before, in this case we do not always have access to components $\mathscr{Y}^v$ globally.  We make use instead of the bundles $\mathscr{L}^{\vec{d}^v}$ which are defined everywhere.  The main advantage is  the following useful property, which follows from the rules governing multidegrees of line bundles on smoothing families described on page \pageref{eq:trivialtensor}: for any $b\in B$, the restriction $\mathscr{L}^{\vec{d}^v}|_{\mathscr{X}_b\times_B T}$ has degree $d$ only on the component of $\mathscr{X}_b$ whose vertex corresponds to the image $cl_{b}(v)$; its degree on any other component vanishes.

We say that the $T$-valued point $[\mathscr{X}\times_B T,\mathscr{D}]$ belongs to $\mathcal{V}_e^{e-f}(\mathscr{X},\ell)(T)$ if, for every $v\in V(\Gamma_0)$ the $(e-f)$-th degeneracy locus of the map
\begin{equation}\label{eq:deglocusblabla}
\mathscr{V}^v \rightarrow (\tau_2)_{*}(\tau_1^*\mathscr{L}^{\vec{d}^v}\otimes\oo_{\mathcal{U}_T})
\end{equation}
is the whole of $T$, where $\tau_1$, $\tau_2$, and $\mathcal{U}_T$ are defined as in the smooth case.  That $\mathcal{V}_e^{e-f}(\mathscr{X},\ell)$ is represented by a closed subscheme of $Sym^e(\mathscr{X}/B)$ proper over $B$ follows from the properties of degeneracy loci, as in Case 1.

Note that if $T=\{b\}$, with $b\in B$ belonging to the image of the smooth locus of $\pi$ and if $f:T\rightarrow B$ is the inclusion, the degeneracy locus condition becomes exactly the one from Section \ref{sec:prelimsec}.  If on the other hand $b\in \Delta_{e'}$, for some edge $e'\in E(\Gamma_0)$ and $f: T\rightarrow B$ is the inclusion, then $\mathscr{X}\times_B T=Y_{(e',v)}\cup Y_{(e',v')}$, with $v,v'$ adjacent vertices to $e'$ and using the notation from Case 0 above. As before, set $\Delta'_{e'}=Y_{(e',v)}\cap Y_{(e',v')}$ and let $D_{(e',v)}$ be the restriction of the divisor $\mathscr{D}$ to $Y_{(e',v)}$.  Then the degeneracy locus condition for the vertex $v$ asks that the divisor $D_{(e',v)}+\Delta'_{e'}$ belongs to the $(e-f)$-th degeneracy locus of the map (\ref{eq:deglocusblabla}) restricted to $Y_{(e',v)}$, where $\Delta'_{e'}$ may appear with multiplicity.  Now let $D_v$ denote the restriction of $D_{(e',v)}$ to the component $Y^v$ corresponding to $v$ and let $\Delta'_{e_1},\ldots,\Delta'_{e_m}$ denote the nodes corresponding to the components of $Y_{(e',v)}$ adjacent to $Y^v$. 
 Using the properties of limit linear series, one sees that the degeneracy locus condition translates to the condition that the divisor $D_v + \Delta'_{e'} + \Delta'_{e_1} + \cdots + \Delta'_{e_m}$ (where the $\Delta'_{e_i}$ and $\Delta'_{e'}$ are understood to occur with the correct multiplicity) of degree $e$ on $Y^v$ belongs to the $(e-f)$-th degeneracy locus of the map (\ref{eq:deglocusblabla}) restricted to $Y^v$.

The dimension bound follows from the degeneracy locus construction of $\mathcal{V}^{e-f}_{e}(\mathscr{X},\ell)$.
\end{proof}

For a linear series $\ell_1$ of type $\grdop{r_1}{d_1}$ on $\mathscr{X}$, denote by $\boldsymbol{\Gamma}_e(\mathscr{X},\ell_1)$ the relative secant variety $\mathcal{V}^{r_1}_{e}(\mathscr{X},\ell_1)$.
Thus in this paper we are interested in the intersection $\boldsymbol{\Gamma}_e(\mathscr{X},\ell_1)\cap\mathcal{V}^{e-f}_{e}(\mathscr{X},\ell_2)$, as we shall see explicitly in what follows.

\section{Intersections of incidence and secant varieties}\label{sec:proofincidsecinter}
In this section we give a proof of Theorem \ref{thm:secant}.  We recall the setup:
consider a complete linear series $l_1=\grdop{r_1}{d_1}$ on a general curve of genus $g$.  We study the intersection of $\Gamma_e(l_1)$ and $V_e^{e-f}(l_2)$, where $l_2=\grdop{r_2}{d_2}=K_C-l_1$ is the residual linear series to $l_1$, when 
\begin{equation}\label{eq:conditionnonexsec}
\dim\Gamma_e(l_1) + \exp \dim V_e^{e-f}(l_2) \leq e-\rho(g,r_1,d_1)-1
\end{equation} 
and prove that it is empty for an arbitrary linear series $l_1\in G^{r_1}_{d_1}(C)$ when $f=1$.

\begin{rem}To get the correct dimensional estimate when we allow for the series $l_1$ to vary in moduli (so we do not consider just the general series of type $\grdop{r_1}{d_1}$), consider the correspondence
\[ \Lambda=\{ (D,l_1)\in C_e\times G^{r_1}_{d_1}(C) \mid D\in \Gamma_e(l_1)\cap V_e^{e-f}(K_C-l_1) \}\subset C_e \times G^{r_1}_{d_1}. \]
By construction, $\Lambda$ has expected dimension
\[ \text{exp}\dim\Lambda = \rho(g,r_1,d_1) + \dim \Gamma_e(l_1) + \dim V^{e-f}_e(K_C-l_1) - e, \]
so if this number is negative, we expect $\Lambda$ to be empty.
\end{rem}

\subsection{The case of minimal pencils}\label{sec:pencil}
Before proving Theorem \ref{thm:secant} in general we first focus on the case of minimal pencils.  This will serve as a prototypical example of the strategy we develop in Section \ref{sec:thmsecant} to check the emptiness of the intersection of incidence and secant varieties 
\[ \Gamma_e(l_1) \cap V^{e-f}_e(K_C-l_1) \]
when condition (\ref{eq:conditionnonexsec}) is satisfied.  We chose to treat this special case separately as the computations are easier to follow than in general and they therefore better illustrate the argument.  In addition, it allows us to discuss a counterexample to the existence of secant divisors, which would otherwise be lost in the analysis.
 
Let $l_1=\grdop{1}{d_1}$ be a minimal pencil, i.e.~such that the Brill-Noether number 
\[ \rho(g,1,d_1)=1. \]
It follows that
\begin{equation}
g=2d_1-3.
\end{equation}
Let $l_2=\grdop{r_2}{d_2}=K_C-l_1=\grdop{d_1-3}{3d_1-8}$.  
Then $\dim \Gamma_e(l_1)=1$ and, as mentioned in the Introduction
\[\exp \dim V^{e-f}_e(K_C-l_1) = e-f(r_2+1-e+f). \]
Thus the non-existence condition (\ref{eq:conditionnonexsec}) of Theorem \ref{thm:secant} becomes
\[1+e-f(r_2+1-e+f)\leq e-2.\]
To ease the computation and presentation, we deal here with the particular case
\begin{equation}\label{eq:0}
1+e-f(r_2+1-e+f)=e-2.
\end{equation}
We show that if (\ref{eq:0}) is satisfied, then the intersection 
\[\Gamma_e(l_1)\cap V^{e-f}_e(l_2)\] is empty.
Condition (\ref{eq:0}) is equivalent to
\[ f(r_2+1-e+f)=3 \]
and we distinguish two possibilities:

\bigskip

\begin{enumerate}[label=\Roman*., wide, labelwidth=!, labelindent=0pt]
	\item If $f=3$, then $r_2-e+f=0$ and $V_e^{e-f}(l_2)=\Gamma_e(l_2)$.  Moreover, 
	\begin{equation}\label{eq:00}
	e=r_2+f=(d_1-3)+3=d_1.
\end{equation}	 	  
	Thus, as expected from the discussion in Section \ref{sec:interincid}, we are in a degenerate situation and we are in fact looking at the inclusion of $l_1=\grdop{1}{d_1}$ inside $l_2=K_C-l_1=\grdop{d_1-3}{3d_1-8}$.  
More precisely, suppose there exists a divisor $D\in C_e$ such that 
\[ D\in\Gamma_e(l_1)\cap \Gamma_e(l_2). \]
	Thus, from (\ref{eq:00}) we have that $|D|=l_1$ and, as we have seen in the proof of Proposition \ref{prop:zerocount}, we have that
	\[|2D + D'| = K_C\] for some effective divisor $D'$ of the correct degree.  More precisely, the condition that $D\in\Gamma_e(l_2)$ is equivalent to
\begin{equation}\label{eq:secancy}
\dim (l_2-D)=\dim (K_C-l_1-D)=\dim |D'| \geq 0.
\end{equation}		
		Since the curve is general, the Petri map
		\[ \mu_0:H^0(C,D)\otimes H^0(C,K_C-D)\rightarrow H^0(C,K_C) \]
		is injective.  Combining this with the base-point-free pencil trick, we get that
		\[ H^0(C,K_C-2D)=H^0(C,D')=0. \]
		This then yields a contradiction with condition (\ref{eq:secancy}).  Hence the intersection $\Gamma_e(l_1)\cap V_e^{e-f}(K_C-l_1)$ is empty in this case.
\begin{rem}\label{rem:interesting}
This actually provides an interesting example that contradicts the expectation of non-emptiness of secant varieties (see Theorem 0.5 in \cite{Fa2}). The inclusion of $l_1=\grdop{1}{d_1}$ in $l_2=\grdop{r_2}{d_2}=\grdop{d_1-3}{3d_1-8}$ can be reformulated from the point of view of secant varieties as follows: there should exist an effective divisor $D'\in C_{2d_1-8}$ such that $\grdop{1}{d_1}+D'=\grdop{d_1-3}{3d_1-8}$.  In other words, the secant variety $V_e^{e-f}(l_2)$, where $e=2d_1-8$ and $f=d_1-4$ should be non-empty and this is indeed the expectation from dimensional considerations as:
\[ e-(r_2+1-e+f)=0. \]
However, as we saw above, there are no such effective divisors $D'$.
\end{rem}
	\item If $f=1$, then $e=d_1-4$ and $r_2-e+f=2$.  
	Assume towards a contradiction that there exists a divisor
	\[D\in\Gamma_e(l_1)\cap V_e^{e-f}(l_2).\]
	  Hence there exists an effective divisor $E\in C_4$ such that $D+E=l_1$.
	Moreover
	\[ l_2-D=K_C-l_1-D=\grdop{r_2-e+f}{2d_1-4}=\grdop{2}{2d_1-4}. \]
	Taking the residue yields
	\[ l_1+D=\grdop{2}{2d_1-4}. \]
	We have therefore obtained a ``system of equations'' for a pair of effective divisors $(D,E)\in C_{d_1-4}\times C_4$:
	\begin{equation}\label{eq:firstsystem}
	\begin{aligned}
	|D+E|&=\grdop{1}{d_1}\\
	|2D+E|&=\grdop{2}{2d_1-4}.
	\end{aligned}
	\end{equation}

By our assumption, a solution for this system exists.
	
A short computation shows that $\rho(g,1,d_1)=1$ implies $\rho(g,2,2d_1-4)=2d_1-12$ which is non-negative if and only if $d_1\geq 6$.  Therefore, in order for the system above to make sense, we let $d_1\geq 6$ from now on.

Furthermore, let $\tilde{l}:=|2D+E|=\grdop{2}{2d_1-4}$.  Hence (\ref{eq:firstsystem}) implies that $D\in V_{d_1-4}^{1}(\tilde{l})$
and we have that
\[ \exp \dim V_{d_1-4}^{1}(\tilde{l}) = -d_1+6. \]
Thus we expect $V_{d_1-4}^{1}(\tilde{l})\neq\emptyset$ for $d_1=6$ and we know from Corollary 0.3 of \cite{Fa2} that if $d_1>6$ and $\tilde{l}$ is general, then $V_{d_1-4}^{1}(\tilde{l})=\emptyset$.  Unfortunately $\tilde{l}$ may not be assumed general in our case.  The non-existence result for secant varieties corresponding to arbitrary linear series is Theorem 0.1 of \cite{Fa2} which states that if
\[\rho(g,2,2d_1-4) +\exp\dim V_{d_1-4}^{1}(\tilde{l}) = d_1-6<0,\]
then $V_{d_1-4}^{1}(\tilde{l})=\emptyset$. Since we let $d_1\geq 6$, this does not impose any further constraints on the existence of solutions to the system (\ref{eq:firstsystem}).
	
  Using a degeneration argument with limit linear series we now show that if the pair $(D,E)$ is a solution to (\ref{eq:firstsystem}) then we have a contradiction.  The idea is to exploit a result of Farkas \cite{Fa2} to obtain a flag curve $\widetilde{R}$ such that all the $d_1$ points coming from the limit of the effective divisors $(D,E)$ specialise to a connected subcurve $Y$ of $\widetilde{R}$ having arithmetic genus at most $d_1$.  This configuration gives rise to ``too much'' ramification at the point of intersection between $Y$ and its complement in $\widetilde{R}$ and the contradiction follows.
	
\bigskip	
	
	We first fix some notation.	
	Following \cite{Fa2}, consider all ``flag curve'' degenerations of curves of genus $g$ that we describe using the ``flag map'' 
\begin{align*}
	j:\overline{\mathcal{M}}_{0,g}&\rightarrow\overline{\mathcal{M}}_g \\
	[R,q_1,\ldots,q_g]&\mapsto[R\cup_{q_1}E_1\cup_{q_2}\ldots\cup_{q_g}E_g]=:[\widetilde{R}],
\end{align*}	
which attaches to each stable curve $[R,q_1,\ldots,q_g]\in\mathcal{M}_{0,g}$ fixed elliptic tails $E_1,\ldots,E_g$ at the points $q_1,\ldots,q_g$, respectively.  Let $p_R:\widetilde{R}\rightarrow R$ be the projection onto $R$, which collapses the elliptic tails i.e.~$p_R(E_i)=q_i$ for $i=1,\ldots,g$.   Denote by $\overline{\mathscr{C}}_{g}=\overline{\mathcal{M}}_{g,1}$ the universal curve over $\overline{\mathcal{M}}_{g}$, and more generally by $\overline{\mathscr{C}}_{g,n}=\overline{\mathcal{M}}_{g,n+1}$ the universal curve over $\overline{\mathcal{M}}_{g,n}$.  Let $\pi_d:\overline{\mathscr{C}}^{d}_{g,n}\rightarrow\overline{\mathcal{M}}_{g,n}$ be the $d$-th fibre product of $\overline{\mathscr{C}}_{g,n}$ over $\overline{\mathcal{M}}_{g,n}$ for some positive integer $d$ and $\overline{\mathcal{M}}_{0,g}\times_{\overline{\mathcal{M}}_g} \overline{\mathscr{C}}^d_g$ the fibre product corresponding to the morphisms $j$ and $\pi_d$.
Finally, we define $\chi$ to be the map that collapses the elliptic tails at the level of the moduli space:
\begin{align*}
\chi: \overline{\mathcal{M}}_{0,g}\times_{\overline{\mathcal{M}}_g} \overline{\mathscr{C}}^d_g &\rightarrow\overline{\mathscr{C}}^d_{0,g}\\
([R,q_1,\ldots,q_g],(y_1,\ldots,y_d))&\mapsto([R,q_1,\ldots,q_g],p_R(y_1),\ldots,p_R(y_d)].
\end{align*}  
	
To apply this setup to our problem, let $W\subset\overline{\mathscr{C}}^{d_1}_{g}$ be the closure of the locus of divisors $(D,E)$ satisfying (\ref{eq:firstsystem}), i.e~the closure of the locus
\[\{ [C,y_1,\ldots,y_{d_1}]\in \mathscr{C}^{d_1}_{g} \mid \exists \grdop{1}{d_1},\,\grdop{2}{2d_1-4}\text{ with }\Bigl|\sum_{i=1}^{d_1} y_i\Bigr|=\grdop{1}{d_1} \text{ and }\Bigl|\grdop{2}{2d_1-4}-\sum_{j=1}^e y_{i_j}\Bigr| = \grdop{1}{d_1} \}.\]
We have assumed that for a general $[C]\in\mathcal{M}_g$ the above locus is non-empty.  This implies that $\pi_{d_1}(W)=\overline{\mathcal{M}}_g$.  Let $U:=\chi(\overline{\mathcal{M}}_{0,g}\times_{\overline{\mathcal{M}}_g}W)\subset\overline{\mathscr{C}}^{d_1}_{0,g}$.  We get therefore that $\pi_{d_1}(U)=\overline{\mathcal{M}}_{0,g}$ and that the minimal fibre dimension of the map $\pi_{d_1}|_U:U\rightarrow\overline{\mathcal{M}}_{0,g}$ is $d_1-m$ for some $0\leq m\leq d_1$.  Hence $\dim(U\cap\pi^{-1}_{d_1}([R,q_1,\ldots,q_g]))\geq d_1-m$ for every point $[R,q_1,\ldots,q_g]\in\overline{\mathcal{M}}_{0,g}$ with equality for a general point $[R,q_1,\ldots,q_g]\in\overline{\mathcal{M}}_{0,g}$.  

	 We now apply Proposition 2.2 of \cite{Fa2}:  let $U'\subset\overline{\mathscr{C}}^{d_1}_{0,g}$ be an irreducible component of the closure of the locus of limits of the divisors $(D,E)$ on flag curves from $\overline{\mathcal{M}}_g$.  Assuming that $\dim U'=g-3+d_1-m$ with $0\leq m\leq d_1$, there exists a point $([R,q_1,\ldots,q_g],(y_1,\ldots,y_d))$ inside $\overline{\mathcal{M}}_{0,g}\times_{\overline{\mathcal{M}}_g}W$ corresponding to a flag curve
	 \[\widetilde{R}:=R\cup E_1\cup\ldots\cup E_g \text{ and points } y_1,\ldots,y_{d_1}\in\widetilde{R}\]
	  such that either:
\begin{enumerate}[label=(\roman*), wide, labelwidth=!, labelindent=0pt]
	\item $y_1=\ldots=y_{d_1}\in R\setminus\{q_1,\ldots,q_g\}$, or else
	\item $y_1,\ldots,y_{d_1}$ lie on a connected subcurve $Y$ of $\widetilde{R}$ of arithmetic genus $p_a(Y)\leq\min(m,g)$ and $|Y\cap\overline{(\widetilde{R}\setminus Y)}|=1$.  
\end{enumerate}

To summarise, if we assume that for a general $[C]\in\mathcal{M}_g$ the system (\ref{eq:firstsystem}) has a solution, then the same should be true for a flag curve $\widetilde{R}$ equipped with limit linear series $\grdop{1}{d_1}$ and $\grdop{2}{2d_1 - 4}$ and such that the points in the support of the limit of the divisors $(D,E)$ satisfy (i) or (ii) above.  

\bigskip

We now obtain the sought after contradiction. 
Case (i) is immediately dismissed via a short computation using the Plücker formula. 

We focus on case (ii).  
Since $g=2d_1 - 3$, it means that $g>d_1\geq m$ for $d\geq 2$, hence without loss of generality we say that the points $y_1,\ldots,y_{d_1}$ lie on a connected subcurve $Y$ with $p_a(Y)=d_1$.  Let $p=Y\cap\overline{(\widetilde{R}\setminus Y)}$ and let $Z:=\overline{\widetilde{R}\setminus Y}$ and let $R_Y$, $R_Z$ denote the rational spines corresponding to $Y$ and $Z$, respectively.

By assumption, $[\widetilde{R},y_1,\ldots,y_{d_1}]\in W$, so there exists a flat, proper morphism $\pi:\mathscr{X}\rightarrow B$ such that:
\begin{itemize}
	\item $\mathscr{X}$ is a smooth surface and $B$ is the spectrum of a discrete valuation ring with uniformising parameter $t$.  
	Moreover, in the notation of Section \ref{sec:degenerationssecant}, the special fibre $\mathscr{X}_0=\pi^{-1}(0)$ is a curve stably equivalent to $\widetilde{R}$ while the general fibre $\mathscr{X}_{\eta}$ is a smooth projective curve of genus $g$.  Finally, there are $d_1$ sections $\sigma_i:B\rightarrow\mathscr{X}$ such that the $\sigma_i(0)=y_i$ are smooth points of $\mathscr{X}_0$ for all $1\leq i \leq d_1$.  Without loss of generality, let $\sigma_1,\ldots,\sigma_e$ be the sections corresponding to the divisor $D$.
	\item $\mathscr{X}_{\eta}$ is equipped with a series $(\mathscr{L}_{\eta},\mathscr{V}_{\eta})$ of type $\grdop{2}{2d_1-4}$.  Furthermore,
\[ \dim\mathscr{V}_{\eta}\cap H^0\Bigl( \mathscr{X}_{\eta},\mathscr{L}_{\eta}\Bigl(-\sum_{j=1}^e\sigma_j(\eta)\Bigr) \Bigr)=2. \]	
\end{itemize}
    
As explained in Section \ref{sec:degenerationssecant}, after possibly making a base change and resolving any resulting singularities, the pair $(\mathscr{L}_{\eta},\mathscr{V}_{\eta})$ induces a refined limit linear series of type $\grdop{2}{2d_1-4}$ on $\widetilde{R}$, which we denote by $\tilde{l}$.  Moreover, the vector bundle 
\[\mathscr{V}_{\eta}\cap\pi_*\Bigl(\mathscr{L}_{\eta}\otimes \oo_{\mathscr{X}_{\eta}}\Bigl(-\sum_{j=1}^e \sigma_j(\eta)\Bigr)\Bigr)\] induces a refined limit linear series $l_1=\grdop{1}{d_1}$ on $\mathscr{X}_0$.  

For a component $X$ of $\mathscr{X}_0$, denote by $(\mathscr{L}_X,\mathscr{V}_X)\in G^{2}_{2d_1-4}(X)$ the $X$-aspect of $\tilde{l}$.  There exists therefore a unique effective divisor $D_X$ of degree $e$ supported only at the points of $(X\cap \bigcup_{j=1}^e \sigma_j(B))\cup (X\cap\overline{\mathscr{X}_0\setminus X})$ such that the $X$-aspect of $l_1$ is of the form
\[ l_{1,X} = (\mathscr{L}_X\otimes\oo_X(-D_X), W_X\subset\mathscr{V}_X\cap H^0(X,\mathscr{L}_X\otimes\oo_X(-D_X)) )\in G^1_{d_1}(X). \]
The collection of aspects $\{ l_{1,X} \}_{X\subset Y}$ forms a refined limit $(l_1)_Y$ of type $\grdop{1}{d_1}$ on $Y$ with a vanishing sequence that is a subsequence of the vanishing sequence of $\tilde{l}$.  We call $\tilde{l}_Y$ the limit linear series induced by $\tilde{l}$ on $Y$.
Moreover, the collection of aspects of $l_1$ on $Z$ also yield a refined limit linear $\grdop{1}{d_1}$ on $Z$ whose vanishing sequence at $p$ is a subsequence of the one of $\tilde{l}$.
 Furthermore, we obtain refined limits of $l_1$ and $\tilde{l}$ on both $R_Y$ and $R_Z$ which we call $(l_1)_{R_Y}$, $(l_1)_{R_Z}$ (of type $\grdop{1}{d_1}$) and $\tilde{l}_{R_Y}$, $\tilde{l}_{R_Z}$ (of type $\grdop{2}{2d_1-4}$), respectively.
 
To reach the desired contradiction, we obtain various bounds for the ramification sequences of the series $l_1$ and $\tilde{l}$ and show that they cannot be simultaneously satisfied.

Note that the points of attachment $q_1,\ldots,q_g$ of the elliptic tails to the rational spine are all cusps, hence for $j=1,\ldots,g$, 
\begin{align}
&\alpha((l_1)_{R_Y},q_j) \geq (0,1) \text{ and } \alpha((l_1)_{R_Z},q_j) \geq (0,1), \label{eq:01}\\
&\alpha(\tilde{l}_{R_Y},q_j) \geq (0,1,1) \text{ and } \alpha(\tilde{l}_{R_Z},q_j) \geq (0,1,1).\label{eq:02}
\end{align}
Moreover, using the Plücker formula (\ref{thm:pluckerlls}) on $R_Y$ we have
\begin{align}
\text{for }l_1=\grdop{1}{d_1}: &\sum_{q\text{ smooth point}}\bigl(\alpha_0((l_1)_{R_Y},q)+\alpha_1((l_1)_{R_Y},q) \bigr) = 2d_1-2, \label{eq:03}\\
\text{for }\tilde{l}=\grdop{2}{2d_1-4}: &\sum_{q\text{ smooth point}}\bigl(\alpha_0(\tilde{l}_{R_Y},q)+\alpha_1(\tilde{l}_{R_Y},q) \bigr) = 6d_1-18. \label{eq:04}
\end{align}
Combining (\ref{eq:01}), (\ref{eq:03}), and (\ref{eq:04}) we obtain that on $R_Y$ the ramification at $p$ is at most
\begin{align}
\text{for }l_1 &: \alpha_0((l_1)_{R_Y},p) + \alpha_1((l_1)_{R_Y},p) \leq d_1-2, \\
\text{for }\tilde{l} &: \sum_{i=0}^2 \alpha_i(\tilde{l}_{R_Y},p) \leq 4d_1-18,
\end{align}
while on $R_Z$ we have the upper bounds
\begin{align}
\text{for }l_1 &: \alpha_0((l_1)_{R_Z},p) + \alpha_1((l_0){R_Z},p) \leq d_1+1, \label{eq:05}\\
\text{for }\tilde{l} &: \sum_{i=0}^2 \alpha_i(\tilde{l}_{R_Z},p) \leq 4d_1-12.
\end{align}
A further constraint for the ramification sequence at $p$ is given by the following vanishing conditions: 

\begin{itemize}
	\item If $\{\sigma_C \mid C\subseteq Y \text{ irreducible component}\}$ is the set of compatible sections corresponding to the divisor $D+E$ and if $p\in C$, then $\ord_p(\sigma_C)=0$.
	\item Similarly, the compatible sections $\{\sigma_C \mid C\subseteq Y \text{ irreducible component}\}$ corresponding to the divisor $2D+E$ also have the property that, if  $p\in C$, then $\ord_p(\sigma_C)=0$. 
\end{itemize}

The important observation in both cases is that the supports of $D+E$ and of $2D+E$ are contained in $Y$ and that $\deg(D+E)=d_1=\deg l_1$ and $\deg(2D+E)=2d_1-4=\deg\tilde{l}$. For a detailed proof, see Lemma 5.2 of \cite{Un}.  Concretely, this means that the vanishing sequences at $p$ of both $(l_1)_{Y}$ and $\tilde{l}_{Y}$ must have 0 as their first entry and consequently  also those of both $(l_1)_{R_Y}$ and $\tilde{l}_{R_Y}$.

Combining this with the compatibility conditions for the vanishing of the sections (\ref{eq:vanish}) and the fact that vanishing sequence at $p$ of $l_1$ is a subsequence of the one of $\tilde{l}$ we see that the only possibility for the vanishing sequences at $p$
of $l_1$ is
\[a((l_1)_{R_Y},p)=(0,d_1-4)  \text{ and } a((l_1)_{R_Z},p)=(4,d_1)\]
and for the vanishing sequence of $\tilde{l}$ at $p$ is
\[a(\tilde{l}_{R_Y},p)=(0,d_1-4,2d_1-8)  \text{ and } a(\tilde{l}_{R_Z},p)=(4,d_1,2d_1-4).\]
However the ramification sequence corresponding to the vanishing sequence \[a((l_1)_{R_Z},p)=(4,d_1)\] is \[\alpha((l_1)_{R_Z},p)=(4,d_1-1),\] which certainly breaks the upper bound in (\ref{eq:05}) and we have obtained the desired contradiction.
\end{enumerate}

\subsection{Proof of Theorem \ref{thm:secant}}\label{sec:thmsecant}
This section is dedicated to proving Theorem \ref{thm:secant}, which states that for any linear series $l_1=\grdop{r_1}{d_1}$ on a general curve $C$ there are no divisors $D\in C_e$ in the intersection
\[ \Gamma_e(l_1)\cap V^{e-f}_e(K_C-l_1) \]
whenever $f=1$, and
\begin{equation}\label{eq:nonexcondition}
\dim\Gamma_e(l_1) + \exp\dim V_e^{e-f}(l_2) \leq e-\rho(g,r_1,d_1)-1.
\end{equation}

In fact we give a general method to check this non-existence statement and apply it to the case $f=1$ where the computations are most tractable.

For the linear series $l_1=\grdop{r_1}{d_1}$ on a general curve $C$ of genus $g$, set \[\rho:=\rho(g,r_1,d_1)\geq 0.\]  Then we have an expression of the genus $g$ in terms of $\rho$:
\begin{equation}\label{eq:genusformula}
g=\frac{(r_1+1)d_1-\rho}{r_1} - r_1 - 1.
\end{equation}
Moreover, an easy computation shows that the residual linear series to $l_1$ is $l_2=\grdop{r_2}{d_2}$ where
\begin{align}
r_2&=\frac{d_1-\rho}{r_1}-2 \label{eq:residualr} \\
d_2&=\frac{(r_1+2)d_1-2\rho}{r_1}-2r_1-4. \label{eq:residuald}
\end{align}
The non-existence condition (\ref{eq:nonexcondition}) becomes
\[ r_1 + e -f(r_2+1-e+f)\leq e-1-\rho, \]
or equivalently
\begin{equation}\label{eq:incidnonex}
f(r_2+1-e+f)\geq r_1+1+\rho,
\end{equation}  
where we used $\dim\Gamma_e(l_1)=r_1$ and $\exp\dim V_e^{e-f}(l_2)=e-f(r_2+1-e+f)$.

Assume towards a contradiction that there exists a divisor $D\in C_e$ such that
\[D\in\Gamma_e(l_1)\cap V_e^{e-f}(l_2).\]
It follows that we also have a divisor $E=l_1-D\in C_{d_1-e}$.  Then 
\[l_2-D=K_C-l_1-D\] is a linear series of dimension
$r_2-e+f$
and degree
\[ \frac{(r_1+2)d_1-2\rho}{r_1}-2r_1-4-e. \]
By residuation we conclude that
\begin{equation}\label{eq:l1plusd}
l_1+D=\grdop{r_1+f}{d+e}.
\end{equation}
We have therefore obtained a ``system of equations'' for two divisors $(D,E)\in C_e \times C_{d_1-e}$:
\begin{equation}\label{eq:system}
\begin{aligned}
|D+E|&=\grdop{r_1}{d_1},\\
|2D+E|&=\grdop{r_1+f}{d_1+e},
\end{aligned}
\end{equation}
and by assumption a solution should exist.  We impose further that the Brill-Noether number of $\tilde{l}:=|2D+E|=\grdop{r_1+f}{d_1+e}$ is also non-negative.

We may view the condition $|2D+E|=\grdop{r_1+f}{d_1+e}$ also from the point of view of de Jonqui\`eres divisors: the dimension of the space of pairs $(D,E)$ satisfying this is
\[ d_1-(d_1+e)+(r_1+f) = r_1-e+f\geq 0, \]
hence no contradiction is detected.  On the other hand, we see that $D\in V_e^{f}(\tilde{l})$ which has
\[ \exp\dim V_e^{f}(\tilde{l}) = f(r_1+1)-r_1 e.\]
Just like in the minimal pencil case, the non-existence condition 
\[\rho(g,r_1+f,d_1+e) + \exp\dim V_e^{f}(\tilde{l}) < 0\]
of \cite{Fa2} for secant varieties corresponding to arbitrary linear series does not impose further restrictions.

We may therefore still assume that there exists a pair of divisors $(D,E)\in C_e \times C_{d_1-e}$ satisfying the system (\ref{eq:system}) and we now produce a contradiction.  Assume furthermore that $g>d_1$ (we shall see later that in the case $f=1$ this assumption does not lead to any loss of generality). We consider again all flag curve degenerations as in the case of minimal pencils and let $W$ be the closure in $\overline{\mathscr{C}}_g^{d_1}$ of the locus
\[\{ [C,y_1,\ldots,y_{d_1}]\in \mathscr{C}^{d_1}_{g} \mid \exists \grdop{r_1}{d_1},\,\grdop{r_1+f}{d_1+e}\text{ with }\Bigl|\sum_{i=1}^{d_1} y_i\Bigr|=\grdop{r_1}{d_1} \text{ and }\Bigl|\grdop{r_1+f}{d_1+e}-\sum_{j=1}^e y_{i_j}\Bigr| = \grdop{r_1}{d_1} \}.\]
Applying Proposition 2.2 of \cite{Fa2}, there exists a point $[\widetilde{R}:=R\cup E_1\cup\ldots\cup E_g,y_1,\ldots,y_{d_1}]\in W$, where $R$ is a rational spine (not necessarily smooth) and the $E_i$ are elliptic tails such that either:
\begin{enumerate}[label=(\roman*), wide, labelwidth=!, labelindent=0pt]
	\item $y_1=\ldots=y_{d_1}$, or else
	\item $y_1,\ldots,y_{d_1}$ lie on a connected subcurve $Y$ of $\widetilde{R}$ of arithmetic genus $p_a(Y)=d_1$ and $|Y\cap\overline{(\widetilde{R}\setminus Y)}|=1$.  This is possible since we have taken $g> d_1$.
\end{enumerate}

Case (i) immediately leads to a contradiction via a short computation using the Plücker formula. 

We focus on case (ii).
Let $p=Y\cap\overline{(\widetilde{R}\setminus Y)}$ and let $Z:=\overline{\widetilde{R}\setminus Y}$ and let $R_Y$, $R_Z$ denote the rational spines corresponding to $Y$ and $Z$, respectively.
Just as in the case of minimal pencils, we have refined limit linear series $l_1$ of type $\grdop{r_1}{d_1}$ and $\tilde{l}$ of type $\grdop{r_1+f}{d_1+e}$ on $\widetilde{R}$ and hence on both $R_Y$ and $R_Z$.

The strategy again is to constrain the vanishing (or, equivalently, ramification) sequence at $p$ of the limit linear series $l_1$ and $\tilde{l}$ on each of the components $R_Y$ and $R_Z$.  In the same notation as for minimal pencils, we make use of four important facts:
\begin{enumerate}[wide, labelwidth=!, labelindent=0pt]
	\item For refined limit linear series, the vanishing sequences at the point $p$ must satisfy the following equalities:
	\begin{equation}\label{eq:vanishingcond}
	\begin{aligned}
	&a_i((l_1)_{R_Y},p) + a_{r_1-i}((l_1)_{R_Z},p)=d_1\text{ for }i=0,\ldots,r_1,\\
	&a_i(\tilde{l}_{R_Y},p) + a_{r_1+f-i}(\tilde{l}_{R_Z},p)=d_1+e\text{ for }i=0,\ldots,r_1 + f.
	\end{aligned}
	\end{equation}	
	\item The vanishing sequence at $p$ of $l_1=\grdop{r_1}{d_1}$ is a subsequence of the one corresponding to $\tilde{l}=\grdop{r_1+f}{d_1+e}$.
	\item The Plücker formula (\ref{thm:pluckerlls}) applied to both limit linear series on both components.
The Plücker formula on $R_Y$ yields:
\begin{align}
\text{for }l_1:  &\sum_{q\text{ smooth point of }R_Y}\biggl(\sum_{i=0}^{r_1} \alpha_i((l_1)_{R_Y},q) \biggr)=(r_1+1)(d_1-r_1)\label{eq:1} \\
\text{for }\tilde{l}: &\sum_{q\text{ smooth point of }R_Y}\biggl(\sum_{i=0}^{r_1+f} \alpha_i(\tilde{l}_{R_Y},q) \biggr)=(r_1+f+1)(d_1+e-f-r_1).\label{eq:2}
\end{align}
The curve $R_Y$ contains the points $q_1,\ldots,q_{d_1}$ which are all cusps, and therefore have ramification sequences at least $(0,1,\ldots,1)$.  Combining this with (\ref{eq:1}) and (\ref{eq:2}) we obtain upper bounds for the ramification at $p$:
\begin{align}
\text{for }l_1: &\sum_{i=0}^{r_1}\alpha_i((l_1)_{R_Y},p)\leq (r_1+1)(d-r_1)-d_1 r_1 \label{eq:3} \\
\text{for }\tilde{l}: &\sum_{i=0}^{r_1+f}\alpha_i(\tilde{l}_{R_Y},p)\leq (r_1+f+1)(d_1+e-f-r_1)-(f+r_1)d_1. \label{eq:4}
\end{align}
Using the same reasoning on $R_Z$ we obtain the following bounds on the ramification at $p$:
\begin{align}
\text{for }l_1: &\sum_{i=0}^{r_1}\alpha_i((l_1)_{R_Z},p)\leq (r_1+1)(d-r_1)-(g-d_1) r_1\label{eq:5} \\
\text{for }\tilde{l}: &\sum_{i=0}^{r_1+f}\alpha_i(\tilde{l}_{R_Z},p)\leq (r_1+f+1)(d_1+e-f-r_1)-(f+r_1)(g-d_1). \label{eq:6}
\end{align}

Since for a linear series $l$ of type $\grd$, 
\begin{equation}
\sum_{i=0}^r \alpha_i(l,p) =  \sum_{i=0}^r a_i(l,p) - \frac{r(r+1)}{2}, \label{eq:7}
\end{equation}
the upper bounds for the ramification give equivalently bounds for the vanishing at $p$. 

\item The statement of Lemma 5.2 of \cite{Un} applied to the current situation, as in the case of the minimal pencils.  We again obtain that both the vanishing sequence of $(l_1)_{R_Y}$ and that of $\tilde{l}_{R_Y}$ must have 0 as their first entry.
\end{enumerate}

Putting everything together, the vanishing sequence at $p$ corresponding to $l_1$ on $R_Y$  is
\[ a((l_1)_{R_Y},p) = (0,x_1,\ldots,x_{r_1}), \]
for some strictly positive integers $x_1,\ldots,x_{r_1}$ smaller than $d_1$, while the sequence on $R_Z$ is
\[ a((l_1)_{R_Z},p) = (d_1-x_{r_1},\ldots,d_1-x_1,d_1). \]
On the other hand, the vanishing sequence at $p$ corresponding to $\tilde{l}$ on $R_Y$  is
\[ a(\tilde{l}_{R_Y},p) = (0,y_1,\ldots,y_{r_1+f}), \]
where the strictly positive integers $y_i$, with $i=1,\ldots,r_1 +f$, are all smaller than $d_1+e$.  Moreover, since the vanishing subsequences of $l_1$ are subsequences of those of $\tilde{l}$, then exactly one of the $y_i$ is equal to $e$ and for each index $i=1,\ldots,r_1$ there exists an index $j$ such that $x_i=y_j$.  Finally, the vanishing sequence at $p$ of $\tilde{l}$ on $R_Z$ is
\[ a(\tilde{l}_{R_Z},p) = (d_1+e-y_{r_1+f},\ldots,d_1,\ldots,d_1+e), \]
which must also contain the terms $d_1-x_{r_1},\ldots,d_1-x_1$.

Let $x=x_1 + \ldots + x_{r_1}$.  Using (\ref{eq:3}), (\ref{eq:5}), and (\ref{eq:7}) and the fact that
\begin{equation}\label{eq:genusineq}
g-d_1 = \frac{d_1-\rho}{r_1}-r_1-1,
\end{equation}
we have that
\begin{equation}\label{eq:8}
r_1\left( \frac{d_1}{r_1} - \frac{r_1+1}{2} \right) - \rho \leq x \leq r_1\left( \frac{d_1}{r_1} - \frac{r_1+1}{2} \right).
\end{equation}

\bigskip

In order to prove the non-existence statement, one finds a contradiction to the inequality (\ref{eq:8}).  As mentioned before, we restrict ourselves to the case $f=1$.

We now have $\tilde{l}=\grdop{r_1+1}{d_1+e}$ and  (\ref{eq:incidnonex}) yields
\begin{equation}\label{eq:expressionfore}
e\leq r_2-r_1-\rho+1=\frac{d_1-(r_1+1)\rho}{r_1}-r_1-1.
\end{equation} 

Moreover, from (\ref{eq:genusineq}) and (\ref{eq:expressionfore}) we get that if $f=1$, then $0<e\leq g-d_1$.  Hence the assumption $g>d_1$ needed in general in order to have a proper subcurve $Y\subsetneq\widetilde{R}$ with $p_a(Y)=d_1<g$ is superfluous in this case.

Suppose first that none of the $x_i$ with $i=1,\ldots,r_1$ is equal to $e$.  Thus the vanishing sequence at $p$ corresponding to $\tilde{l}$ on $R_Y$ is
\[ a(\tilde{l}_{R_Y},p) = (0,e,x_1,\ldots,x_{r_1}),\]
up to a permutation of the terms $e,x_1,\ldots,x_{r_1}$.  We note that the exact order of the terms in the vanishing sequence does not matter in the arguments below, as we are always considering the sums of their entries.  
Combining (\ref{eq:6}) and (\ref{eq:7}) yields the inequality
\begin{align*}
(r_1+2)(d_1+e)-e-x-\frac{(r_1+1)(r_1+2)}{2}\leq &(r_1+2)(d_1+e-1-r_1)\\
&-(r_1+1)\left( \frac{d_1-\rho}{r_1}-r_1-1 \right)
\end{align*}  
which, after plugging in the expression (\ref{eq:expressionfore}) for $e$, reduces to
\[ x\geq \frac{(r_1+1)(r_1+2)}{2} + (r_1+1)\left( \frac{d_1}{r_1}-r_1-1 \right). \]
This contradicts the upper bound in (\ref{eq:8}).  Hence this vanishing sequence cannot occur.

One the other hand, if $e$ is one of the $x_i$ with $i=1,\ldots,r_1$, then, abusing notation as before, the vanishing sequence 
at $p$ corresponding to $l_1$ on $R_Y$ is
\[ a((l_1)_{R_Y},p) = (0,e,x_1,\ldots,x_{r_1-1}) \]
and on $R_Z$
\begin{equation}\label{eq:9}
a((l_1)_{R_Z},p) = (d_1-x_{r_1-1},\ldots,d_1-x_1,d_1-e,d_1). 
\end{equation} 
Moreover, the vanishing sequence at $p$ corresponding to $\tilde{l}$ on $R_Y$ is
\[ (0,e,x_1,\ldots,x_{r_1-1},y), \]
for some positive integer $y$, and the one on $R_Z$ is
\begin{equation}\label{eq:10}
(d_1+e-y,d_1+e-x_{r_1-1},\ldots,d_1 +e -x_1,d_1,d_1+e).
\end{equation}  
Since the sequence (\ref{eq:9}) must be a subsequence of (\ref{eq:10}), we see that
\[d_1+e-y=d_1-x_i,\] for some index $i$.  In other words, $y=e+x_i$. 
Combining (\ref{eq:6}) and (\ref{eq:7}) again yields the inequality
\begin{align*}
(r_1+2)(d_1+e)-e-x_i-x-\frac{(r_1+1)(r_1+2)}{2}&\leq (r_1+2)(d_1+e-1-r_1)\\
&-(r_1+1)\left( \frac{d_1-\rho}{r_1}-r_1-1 \right).
\end{align*}  
This leads to a contradiction with the upper bound in (\ref{eq:8}) in the same way as above.
 

 \end{document}